\documentclass[reqno]{amsart}
\usepackage{amssymb,latexsym,amsmath,amsthm,enumerate}
\usepackage[mathscr]{eucal}
\usepackage{framed,color,graphicx}
\usepackage{mathrsfs}
\usepackage[all]{xy}
\usepackage{tikz}
\usepackage{cite}

\usepackage{longtable}
\usepackage{ltxtable}
\usepackage{subcaption}
\usepackage{graphicx}
\usepackage{subfloat}
\usepackage{subfig}

\usepackage{subcaption}

\makeatletter
\@namedef{subjclassname@2020}{\textup{2020} Mathematics Subject Classification}
\makeatother

\usetikzlibrary{positioning}
\tikzset{%
element/.style={draw, shape=circle, fill=white, inner sep=1.4pt}
}
\DeclareSymbolFont{bbold}{U}{bbold}{m}{n}
\DeclareSymbolFontAlphabet{\mathbbold}{bbold}

\theoremstyle{plain}
\newtheorem{thm}{Theorem}[section]
\newtheorem{lem}[thm]{Lemma}
\newtheorem{cor}[thm]{Corollary}
\newtheorem{pro}[thm]{Proposition}

\newtheorem{problem}[thm]{Problem}

\newtheorem{claim}{Claim}[section]

\theoremstyle{definition}

\newtheorem{remark}[thm]{Remark}

\newcommand{\up}[1]{\textup{#1}}

\newcommand{\bp}{\mathbf{p}}
\newcommand{\bq}{\mathbf{q}}
\newcommand{\br}{\mathbf{r}}
\newcommand{\bs}{\mathbf{s}}
\newcommand{\bt}{\mathbf{t}}
\newcommand{\bu}{\mathbf{u}}
\newcommand{\bv}{\mathbf{v}}
\newcommand{\bw}{\mathbf{w}}

\begin{document}
\title[Power semirings of $S_7$]
{Explicit equational bases for the power semirings of $S_7$}

\author{Mengya Yue}
\address{School of Mathematics, Northwest University, Xi'an, 710127, Shaanxi, P.R. China}
\email{myayue@yeah.net}

\author{Miaomiao Ren}
\address{School of Mathematics, Northwest University, Xi'an, 710127, Shaanxi, P.R. China}
\email{miaomiaoren@yeah.net}

\author{Zidong Gao}
\address{School of Mathematics, Northwest University, Xi'an, 710127, Shaanxi, P.R. China}
\email{zidonggao@yeah.net}

\subjclass[2020]{16Y60, 03C05, 08B05, 08B15}
\keywords{Additively idempotent emiring, Variety, Identity, Finite basis problem}
\thanks{Miaomiao Ren, corresponding author, is supported by National Natural Science Foundation of China (12371024, 12571020).
Mengya Yue is supported by the Research Innovation Project for Postgraduates of Northwest University (CX2026052).
}

\begin{abstract}
For every semigroup $S$, the set $\mathcal{P}(S)$ of all subsets of $S$ and the set $\mathcal{P}^{+}(S)$
of all nonempty subsets of $S$ form additively idempotent semirings under set-theoretic union and
elementwise multiplication, called the full and nonempty power semirings of $S$, respectively.
We investigate the finite basis problem for the full and nonempty power semirings $\mathcal{P}(S_7)$ and $\mathcal{P}^{+}(S_7)$ of the multiplicative reduct of $S_7$,
where $S_7$ is the unique nonfinitely based three-element additively idempotent semiring.
We provide explicit infinite equational bases for both and prove that they are nonfinitely based.
For $\mathcal{P}^{+}(S_7)$, we establish a new sufficient condition for an additively idempotent semiring to be
nonfinitely based and apply it to obtain the required result.
Moreover, we show that the interval
$[\mathsf{V}(\mathcal{P}^{+}(S_7)), \mathsf{V}(\mathcal{P}(S_7))]$
in the lattice of additively idempotent semiring varieties has the cardinality of the continuum.
\end{abstract}

\maketitle

\section{Introduction}
An additively idempotent semiring (or ai-semiring for short) is an
algebra $(S,+,\cdot)$ such that the additive reduct $(S,+)$ is a
commutative idempotent semigroup, the multiplicative reduct
$(S,\cdot)$ is a semigroup, and multiplication distributes over
addition from both sides:
\[
x(y+z)\approx xy+xz
\quad\text{and}\quad
(x+y)z\approx xz+yz.
\]
An ai-semiring whose multiplicative reduct is commutative is called a \emph{commutative ai-semiring}.

The class of ai-semirings includes
well-known examples such as the Kleene semiring of regular languages~\cite{con},
the max-plus algebra~\cite{aei}, the power semiring of a semigroup~\cite{dgv24},
the endomorphism semiring of a semilattice~\cite{dgv25},
the semiring of all binary relations on a set \cite{dolinka}, and distributive lattices~\cite{bs}.
These and other similar algebras have significant applications in various fields,
including algebraic geometry~\cite{cc}, tropical algebraic geometry~\cite{ms},
theoretical computer science~\cite{go}, information science~\cite{gl}, and the theory of weighted automata~\cite{dkv}.

A class of ai-semirings is a \emph{variety} if it is closed under taking subalgebras,
homomorphic images, and arbitrary direct products.
By Birkhoff's theorem, a class of ai-semirings is a variety if and only if
it is an equational class, that is, the class of all ai-semirings satisfying a certain set of identities.
Let $\mathcal{V}$ be an ai-semiring variety.
If $\Sigma$ is a set of identities that defines $\mathcal{V}$,
then $\Sigma$ is an \emph{equational basis} of $\mathcal{V}$.
The variety $\mathcal{V}$ is \emph{finitely based} if it admits a finite equational basis;
otherwise, it is \emph{nonfinitely based}.
A variety is \emph{locally finite} if every finitely generated
algebra in it is finite. A finite algebra $A$ is
\emph{inherently nonfinitely based} if it is not contained
in any finitely based locally finite variety.
A finite ai-semiring is \emph{strongly nonfinitely based}
if every finite ai-semiring whose generated variety contains it
is nonfinitely based.

For an ai-semiring $S$, let $\mathsf{V}(S)$ denote the variety generated by $S$,
that is, the smallest variety containing $S$.
Then $S$ and $\mathsf{V}(S)$ satisfy exactly the same identities.
If $\Sigma$ is an equational basis of $\mathsf{V}(S)$, we also say that $\Sigma$ is an equational basis of $S$.
The ai-semiring $S$ is called finitely based or nonfinitely based according to whether $\mathsf{V}(S)$ is finitely based or not.

The \emph{finite basis problem} for a class of ai-semirings,
one of the central questions in the theory of varieties,
concerns the classification of its members with respect to the finite basis property.
While we focus here on ai-semirings, the finite basis problem for arbitrary algebras
has deep and often surprising connections with
formal languages \cite{Almeida}, the complexity of cellular automata \cite{Gulak} and
classical number-theoretic conjectures \cite{Perkins}.

A natural source of ai-semirings is the powerset construction. Let
$S$ be a semigroup. We denote by $\mathcal{P}(S)$ the collection of
all subsets of $S$, and by $\mathcal{P}^{+}(S)$ the collection of all
nonempty subsets of $S$. For $A,B\subseteq S$, define
\[
A+B=A\cup B,
\qquad
AB=\{ab\mid a\in A,\ b\in B\}.
\]
Then both $\mathcal{P}(S)$ and $\mathcal{P}^{+}(S)$ are
ai-semirings, called the power semiring and the nonempty power semiring
of $S$, respectively.

Let $S$ be an ai-semiring.
The algebra $S^0 = S \cup \{0\}$ obtained from $S$ by adjoining a new element $0$ is again an ai-semiring under the operations
\[
(\forall a\in S\cup \{0\}) \quad a+0=0+a=a,\quad a0=0a=0,
\]
with the original operations of $S$ unchanged.
Thus $S$ is a subsemiring of $S^0$, and the element
$0$ is simultaneously the additive least element and the multiplicative zero of $S^0$.
Notice that the empty set is both the additive
identity and the multiplicative zero of $\mathcal{P}(S)$. Consequently,
\[
\mathcal{P}(S)\cong
\bigl(\mathcal{P}^{+}(S)\bigr)^{0},
\]
where $T^{0}$ denotes the ai-semiring obtained from $T$ by adjoining a
new element that is simultaneously the additive identity and the
multiplicative zero. Throughout this paper, $\mathcal{P}(S)$ always
includes the empty set, whereas $\mathcal{P}^{+}(S)$ does not.

Over the past two decades, the finite basis problem for ai-semirings has been intensively studied and well developed,
for example, see~\cite{do, dolinka, dgv24, dgv25, gmrz, gpz, gv2501, jrz, pas05, rjzl, rlzc, rlyc, rz16, rzs20, rzw, sr, shap23, Volkov, wrz, yr2601, yrzs, zrc, zw}.
Within this broader program, the finite basis problem for power semirings has also been actively investigated.
Dolinka~\cite[Problem~6.5]{dolinka} and Jackson et al.~\cite[Problem~7.2]{jrz} posed the following general problem.

\begin{problem}\label{power-semiring}
When is the power semiring of a finite semigroup finitely based as an
ai-semiring?
\end{problem}

Several partial answers to Problem~\ref{power-semiring} have subsequently been obtained.
Dolinka~\cite{dolinka} proved that $\mathcal{P}(S)$ is inherently nonfinitely based whenever $S$ is a finite semigroup such that
the five-element Brandt semigroup divides $S$, or whenever $S$ itself is inherently nonfinitely based.
He later proved~\cite{do10} that, for a finite group $G$, the
$(\mathcal{P}(G),\cdot)$ is inherently nonfinitely based precisely when $G$ is not a Dedekind group.
Jackson et al.~\cite{jrz} proved that $\mathcal{P}^{+}(G)$ is strongly nonfinitely based
whenever $G$ is a finite group containing a nonabelian nilpotent subgroup.
Consequently, both $\mathcal{P}^{+}(G)$ and $\mathcal{P}(G)$ are nonfinitely based in this case.

Gusev and Volkov~\cite{gv23} proved that both $\mathcal{P}(G)$ and $\mathcal{P}^{+}(G)$
are nonfinitely based whenever $G$ is a finite nonabelian solvable group.
Dolinka et al.~\cite{dgv24} showed that
$\mathcal{P}(S)$ is nonfinitely based if $S$ is a finite inverse semigroup such that
either $S$ is not Clifford or all subgroups of $S$ are solvable and at least one of them is nonabelian.
Wu and Zhao~\cite{wz} proved that the power semiring $P^+({\dot{S}_c(W)})$ of the finite nil-semigroup $\dot{S}_c(W)$
is nonfinitely based under certain combinatorial conditions on $W$.
Recently, Gao et al.~\cite{grsy} prove that
for a finite group $G$, $\mathcal{P}^{+}(G)$ has no finite basis for its identities if and only if $|G|\geq 3$.

The ai-semiring $S_7$ (its Cayley tables are given in Table~\ref{tb7}) is of particular relevance to the present
paper. Jackson et al.~\cite[Corollary~5.1]{jrz} proved that, up to
isomorphism, $S_7$ is the unique nonfinitely based ai-semiring of
order at most three. Thus, despite its small size and relatively
simple operations, $S_7$ occupies a distinguished position in the
finite basis theory of ai-semirings. Moreover, several
studies~\cite{gmrz,jrz,yr2601} have shown that the nonfinite basis
property of $S_7$ transfers to many other finite ai-semirings whose
generated varieties contain $S_7$.
This observation led Gao et al.~\cite{gjrz2} (a manuscript in preparation, not publicly available)
to pose the following general problem:

\begin{problem}\label{prob0127}
Is every finite ai-semiring whose variety contains $S_7$ nonfinitely based?
\end{problem}

\begin{table}[ht]
\centering
\caption{The Cayley tables of $S_7$}
\label{tb7}
\begin{tabular}{c|ccc}
$+$ & $0$ & $a$ & $1$\\
\hline
$0$ & $0$ & $0$ & $0$\\
$a$ & $0$ & $a$ & $0$\\
$1$ & $0$ & $0$ & $1$
\end{tabular}
\qquad
\begin{tabular}{c|ccc}
$\cdot$ & $0$ & $a$ & $1$\\
\hline
$0$ & $0$ & $0$ & $0$\\
$a$ & $0$ & $0$ & $a$\\
$1$ & $0$ & $a$ & $1$
\end{tabular}
\end{table}

Motivated by the preceding developments on power semirings and by the distinguished role of $S_7$,
in the present paper we investigate the finite basis problem for $\mathcal{P}(S_7,\cdot)$ and $\mathcal{P}^{+}(S_7,\cdot)$.
Their Cayley tables are displayed in Tables~\ref{tbPS7} and~\ref{tbPS7plus}, respectively.

More precisely, we first provide an explicit infinite equational basis for each of $\mathcal{P}(S_7,\cdot)$ and $\mathcal{P}^{+}(S_7,\cdot)$.
We then establish a new sufficient condition for an ai-semiring to be nonfinitely based and apply it to $\mathcal{P}^{+}(S_7,\cdot)$.
The main result of this paper is that both $\mathcal{P}(S_7,\cdot)$ and $\mathcal{P}^{+}(S_7,\cdot)$ are nonfinitely based,
thereby resolving the finite basis problem for these two power semirings.
Consequently, our results provide a further partial answer to Problem~\ref{power-semiring}.

Moreover, we show that
\[
S_7\in
\mathsf{V}\bigl(\mathcal{P}^{+}(S_7,\cdot)\bigr).
\]
Since $\mathcal{P}^{+}(S_7,\cdot)$ is a subsemiring of $\mathcal{P}(S_7,\cdot)$, it follows that
\[
S_7
\in
\mathsf{V}\bigl(\mathcal{P}^{+}(S_7,\cdot)\bigr)
\subseteq
\mathsf{V}\bigl(\mathcal{P}(S_7,\cdot)\bigr).
\]
Thus these two power semirings provide two further examples
supporting an affirmative answer to
Problem~\ref{prob0127}.

The above inclusion naturally raises the question of how far apart the varieties
generated by the full and nonempty power semirings of a semigroup can be,
or, equivalently, how many varieties may lie between them.
This leads to the following general problem.

\begin{problem}\label{interval}
Let $S$ be a semigroup. What is the cardinality of the interval
\[
[\mathsf{V}(\mathcal{P}^{+}(S)), \mathsf{V}(\mathcal{P}(S))]
\]
in the lattice of semiring varieties?
\end{problem}

In the present paper, we answer Problem~\ref{interval} for the multiplicative reduct of $S_7$.
More precisely, we prove that the interval
\[
[\mathsf{V}(\mathcal{P}^{+}(S_7,\cdot)),
  \mathsf{V}(\mathcal{P}(S_7,\cdot))]
\]
contains $2^{\aleph_0}$ distinct varieties and therefore has the cardinality of the continuum.
This provides a partial answer to Problem~\ref{interval}.

\begin{table}[ht]
\centering
\caption{The Cayley tables of $P(S_7, \cdot)$}
\label{tbPS7}
\setlength{\tabcolsep}{3pt}

\begin{tabular}{c|cccccccc}
$+$&$1$&$2$&$3$&$4$&$5$&$6$&$7$&$8$\\
\hline
$1$&$1$&$1$&$1$&$1$&$1$&$1$&$1$&$1$\\
$2$&$1$&$2$&$1$&$1$&$2$&$2$&$1$&$2$\\
$3$&$1$&$1$&$3$&$1$&$3$&$1$&$3$&$3$\\
$4$&$1$&$1$&$1$&$4$&$1$&$4$&$4$&$4$\\
$5$&$1$&$2$&$3$&$1$&$5$&$2$&$3$&$5$\\
$6$&$1$&$2$&$1$&$4$&$2$&$6$&$4$&$6$\\
$7$&$1$&$1$&$3$&$4$&$3$&$4$&$7$&$7$\\
$8$&$1$&$2$&$3$&$4$&$5$&$6$&$7$&$8$
\end{tabular}
\qquad
\begin{tabular}{c|cccccccc}
$\cdot$&$1$&$2$&$3$&$4$&$5$&$6$&$7$&$8$\\
\hline
$1$&$1$&$1$&$3$&$1$&$5$&$1$&$3$&$8$\\
$2$&$1$&$2$&$3$&$1$&$5$&$2$&$3$&$8$\\
$3$&$3$&$3$&$5$&$3$&$5$&$3$&$5$&$8$\\
$4$&$1$&$1$&$3$&$1$&$5$&$4$&$3$&$8$\\
$5$&$5$&$5$&$5$&$5$&$5$&$5$&$5$&$8$\\
$6$&$1$&$2$&$3$&$4$&$5$&$6$&$7$&$8$\\
$7$&$3$&$3$&$5$&$3$&$5$&$7$&$5$&$8$\\
$8$&$8$&$8$&$8$&$8$&$8$&$8$&$8$&$8$
\end{tabular}
\end{table}

\begin{table}[ht]
\centering
\caption{The Cayley tables of $P^{+}(S_7, \cdot)$}
\label{tbPS7plus}
\setlength{\tabcolsep}{3pt}

\begin{tabular}{c|ccccccc}
$+$&$1$&$2$&$3$&$4$&$5$&$6$&$7$\\
\hline
$1$&$1$&$1$&$1$&$1$&$1$&$1$&$1$\\
$2$&$1$&$2$&$1$&$1$&$2$&$2$&$1$\\
$3$&$1$&$1$&$3$&$1$&$3$&$1$&$3$\\
$4$&$1$&$1$&$1$&$4$&$1$&$4$&$4$\\
$5$&$1$&$2$&$3$&$1$&$5$&$2$&$3$\\
$6$&$1$&$2$&$1$&$4$&$2$&$6$&$4$\\
$7$&$1$&$1$&$3$&$4$&$3$&$4$&$7$
\end{tabular}
\qquad
\begin{tabular}{c|ccccccc}
$\cdot$&$1$&$2$&$3$&$4$&$5$&$6$&$7$\\
\hline
$1$&$1$&$1$&$3$&$1$&$5$&$1$&$3$\\
$2$&$1$&$2$&$3$&$1$&$5$&$2$&$3$\\
$3$&$3$&$3$&$5$&$3$&$5$&$3$&$5$\\
$4$&$1$&$1$&$3$&$1$&$5$&$4$&$3$\\
$5$&$5$&$5$&$5$&$5$&$5$&$5$&$5$\\
$6$&$1$&$2$&$3$&$4$&$5$&$6$&$7$\\
$7$&$3$&$3$&$5$&$3$&$5$&$7$&$5$
\end{tabular}
\end{table}

\section{Preliminaries}
In this section, we collect some basic notation, definitions, and tools that will be used throughout the paper.
Although every ai-semiring in the subsequent sections is commutative,
the results quoted below are stated for general ai-semirings (whose multiplication need not be commutative).
This formulation keeps them applicable in future work that may not impose commutativity.

Let $X$ be a countably infinite set of variables,
let $X^+$ denote the free semigroup over $X$,
$X^*$ the free monoid over $X$ (with $\varepsilon$ denoting the empty word),
and $X_c^+$ the free commutative semigroup over $X$.

An \emph{ai-semiring term} (or simply a \emph{term}) over $X$ is a finite nonempty set of words in $X^+$.
(In what follows, terms are denoted by bold lowercase letters \(\mathbf{u}, \mathbf{v}, \mathbf{w}, \dots\),
while ordinary lowercase letters \(x, y, z, \dots\) stand for variables.)
We write a term as a formal sum of its elements; that is,
$\bw=\bu_1+\bu_2+\cdots+\bu_n$ means $\bw=\{\bu_1, \bu_2, \ldots, \bu_n\}$.
The order of the summands in the formal sum is irrelevant,
and repeated occurrences of the same word are identified with a single occurrence.
Two terms are equal if and only if their underlying sets coincide.

The collection of all terms over $X$, denoted by $P_f(X^+)$,
forms an ai-semiring under the usual term addition and multiplication.
By \cite[Theorem 2.5]{kp}, $P_f(X^+)$ is free in the variety of all ai-semirings over $X$.
Although the multiplicative reduct of $P_f(X^+)$ is not cancellative in general,
cancellation does hold in the following cases:
if $\bp$, $\bq$, and $\br$ are terms (at least one of which is a word) such that $\bp\bq=\bp\br$, then $\bq=\br$;
if $\bq\bp=\br\bp$ with at least one of $\bp, \bq, \br$ being a word, then $\bq=\br$.
Furthermore,
for any terms $\bu$ and $\bv$, if $\bu\bv$ is a word, then
both $\bu$ and $\bv$ are themselves words.

Let $\bu$ and $\bv$ be terms.
We say that \emph{$\bu$ is a subterm of $\bv$}, and write $\bu\leq \bv$,
if there exist terms $\bp_1, \bp_2, \bp_3$ such that
\[
\bv=\bp_1\bu\bp_2+\bp_3.
\]
Here $\bp_1$ or $\bp_2$ may be the empty word (acting as the multiplicative identity),
and $\bp_3$ may be the empty set (acting as the additive identity).
In particular, if $\bp_1$ and $\bp_2$ are both empty, then $\bu$ is called an \emph{additive subterm} of $\bv$;
if $\bp_3$ is empty, then $\bu$ is called a \emph{multiplicative subterm} of $\bv$.
Under the subset representation, an additive subterm of $\bv$ corresponds to a nonempty subset of $\bv$.
If $\bv$ is a word and $\bu\leq \bv$,
then $\bu$ must be a multiplicative subterm of $\bv$ and is itself a word;
in this case we say that $\bu$ is \emph{subword} of $\bv$.
Moreover, the relation $\leq$ is a partial order on $P_f(X^+)$.
Although $\leq$ is not compatible with the multiplication of $P_f(X^+)$,
the corresponding relation on $P_f(X^+_c)$ is compatible with its multiplication.

An \emph{ai-semiring substitution} (or simply a \emph{substitution}) is a semiring homomorphism from $P_f(X^+)$ to itself.
The term $\bv$ is called \emph{$\bu$-free} if for every substitution $\varphi$,
the term $\varphi(\bu)$ is not a subterm of $\bv$.
The following statement follows directly from the definitions of subterm and freeness.

\begin{lem}\label{lem26012701}
Let $\bu$, $\bv$ and $\bw$ be terms.
If $\bv$ is $\bu$-free and $\bu$ is a subterm of $\bw$, then $\bv$ is also $\bw$-free.
\end{lem}

An \emph{ai-semiring identity} (or simply an \emph{identity}) is a formal expression of the form
$\bu\approx \bv$, where $\bu$ and $\bv$ are terms.
Let $S$ be an ai-semiring and $\bu\approx \bv$ an identity.
We say that $S$ \emph{satisfies} $\bu\approx \bv$, or that $\bu\approx \bv$ \emph{holds} in $S$,
if $\varphi(\bu)=\varphi(\bv)$ for every semiring homomorphism $\varphi\colon P_f(X^+) \rightarrow S$.

The following result, which concerns the equational logic of ai-semirings, appears in \cite[Lemma 2]{do}.
To the best of our knowledge, this is the first occurrence of this result in the context of ai-semirings.
Moreover, it can be seen as a direct specialization of the universal-algebraic fact \cite[Exercise II.14.11]{bs} to ai-semirings.
\begin{lem}\label{lem02}
Let $\Sigma$ be a set of identities and let $\bu\approx \bv$ be a nontrivial identity.
Then $\bu\approx \bv$ is derivable from $\Sigma$ if and only if there
exist terms $\bt_1, \bt_2, \dots, \bt_n\in P_f(X^+)$ such that $\bu=\bt_1$, $\bv=\bt_n$ and,
for each $1\leq i<n$, there are terms $\bp_i, \bq_i, \br_i, \bs_i, \bs'_{i} \in P_f(X^+)$ and a substitution
$\varphi_i\colon P_f(X^+) \to P_f(X^+)$ such that
\[
\bt_i = \bp_i\varphi_i(\bs_i)\bq_i+\br_i, \quad \bt_{i+1}=\bp_i\varphi_i(\bs'_{i})\bq_i+\br_i,
\]
where $\bs_{i} \approx \bs'_{i}\in\Sigma$ or $\bs'_{i} \approx \bs_{i} \in \Sigma$,
$\bp_i$ or $\bq_i$ may be the empty word \up(acting as the multiplicative identity\up),
and $\br_i$ may  be the empty set \up(acting as the additive identity\up).
\end{lem}

The following lemma, which is the compactness theorem of equational logic (see \cite[Exercise II.14.10]{bs}),
will be our main tool in Sections \ref{sec3} and \ref{sec4} for proving that both $\mathcal{P}(S_7,\cdot)$ and $\mathcal{P}^{+}(S_7,\cdot)$ have no finite equational basis.
\begin{lem}\label{el26012701}
Let $\mathcal{V}$ be an ai-semiring variety.
If $\mathcal{V}$ is finitely based,
then every equational basis of $\mathcal{V}$ contains a finite subset that also defines $\mathcal{V}$.
Equivalently, if $\mathcal{V}$ has an infinite equational basis none of whose finite subsets defines $\mathcal{V}$,
then $\mathcal{V}$ is nonfinitely based.
\end{lem}

Following Volkov~\cite{volkov2024}, we write $\bu \preceq \bv$ (or equivalently $\bv \succeq \bu$)
to denote the identity $\bv \approx \bv+\bu$ and call it an \emph{ai-semiring inequality} (or simply an \emph{inequality}).
The term $\mathbf{u}$ is called the \emph{lower side} and $\mathbf{v}$ the \emph{upper side} of the inequality.
It is easy to verify that an ai-semiring $S$ satisfies an inequality $\bu\preceq \bv$
if and only if $\varphi(\bu)\leq_S \varphi(\bv)$ for every semiring homomorphism $\varphi\colon P_f(X^+)\rightarrow S$.
Consequently, $S$ satisfies an identity $\bu\approx \bv$
precisely when it satisfies both inequalities $\bu\preceq\bv$ and $\bv\preceq\bu$.
For this reason, every set of identities can be equivalently expressed as a set of inequalities.

Suppose that $\Sigma$ is a set of identities.
Let $\mathbf{u}\approx \mathbf{v}$ be an identity such that
\[
\mathbf{u}=\mathbf{u}_1+\cdots+\mathbf{u}_k,\quad \mathbf{v}=\mathbf{v}_1+\cdots+\mathbf{v}_\ell,
\]
where $\mathbf{u}_i,\mathbf{v}_j\in X^+$ for $1\leq i\leq k$ and $1\leq j\leq \ell$.
One can readily check that the ai-semiring variety defined by $\mathbf{u}\approx \mathbf{v}$
coincides with the ai-semiring variety defined by the inequalities
\[
\mathbf{u}_i\preceq \mathbf{v}, \quad \mathbf{v}_j\preceq \mathbf{u}\quad(1\leq i\leq k, ~1\leq j\leq \ell).
\]
Consequently, to prove that $\mathbf{u}\approx \mathbf{v}$ is derivable from $\Sigma$,
it suffices to show that for every $i$ and $j$,
the inequalities
$\mathbf{u}_i\preceq \mathbf{v}$, $\mathbf{v}_j\preceq \mathbf{u}$
can be derived from $\Sigma$.
For this reason, throughout this paper we restrict our attention to the inequalities of the form $\bq\preceq\bu$,
where $\bq$ is a word and $\bu$ is a term.
This technique will be used repeatedly in the sequel.

Next, we introduce some notation.
Let $\bw$ be a nonempty word, and let $x$ be a variable. Then
\begin{itemize}
\item $c(\bw)$ denotes the \emph{content} of $\bw$, that is, the set of all variables that occur in $\bw$;

\item $\ell(\bw)$ denotes the \emph{length} of $\bw$, that is, the number of variables occurring in $\bw$ counting multiplicities;

\item $S_2(\bw)$ denotes the set of all subwords of length $2$ of $\bw$;

\item $m(x, \bw)$ denotes the number of occurrences of $x$ in $\bw$.
\end{itemize}

Now let $\bu$ be a term such that $\bu=\bu_1+\bu_2+\cdots+\bu_n$,
where $\bu_i \in X^+$, $1 \leq i \leq n$.
Let $\bq$ be a nonempty word, and let $k$ be a positive integer. Then
\begin{itemize}
\item $c(\bu)$ denotes the \emph{content} of $\bu$, that is,
\[
c(\bu)=\bigcup_{i=1}^n c(\bu_i);
\]

\item $L_{\geq k}(\bu)$ denotes the set $\{\bu_i \in \bu \mid \ell(\bu_i)\geq k\}$;

\item $D_{\bq}(\bu)$ denotes the set $\{\bu_i \in \bu \mid c(\bu_i)\subseteq c(\bq)\}$;

\item $S_2(\bu)$ denotes the set $\bigcup_{1\leq i \leq n}S_2(\bu_i)$;
%
%\item $\delta(\bu)$ denotes the set of subsets $Z$ of $c(\bu)$ such that for every
% $\bu_i\in\bu$, $Z\cap c(\bu_i)$ is a singleton and $occ(x,\bu_i)=1$ if $\{x\}=Z\cap c(\bu_i)$.
\end{itemize}

Finally, we state the characterization of identities of $D_2$, which follows directly from Shao and Ren~\cite[Lemma 1.1]{sr}.
\begin{lem}\label{lem01}
Let $\bq\preceq\bu$ be a nontrivial inequality such that
$\bu=\bu_1+\bu_2+\cdots+\bu_n$ and $\bu_i, \bq \in X^+$ for all $1\leq i \leq n$. Then
$\bq\preceq\bu$ holds in $D_2$ if and only if $c(\bu_i) \subseteq c(\bq)$ for some $\bu_i \in \bu$.
\end{lem}

\section{Equational basis for $\mathcal{P}(S_7)$}\label{sec3}
In this section, we provide an explicit infinite equational basis for $\mathcal{P}(S_7,\cdot)$
and prove that its identities admit no finite basis.
It is readily verified that $\mathcal{P}(S_7,\cdot)$ is isomorphic to a subdirect product of $S_{53}^{0}$ and $S_{(4,450)}$.
Their Cayley tables are displayed in Tables~\ref{tb634} and~\ref{tb450}, respectively.
Hence,
\[
\mathsf{V}\bigl(\mathcal{P}(S_7,\cdot)\bigr)
=
\mathsf{V}\bigl(S_{53}^{0},S_{(4,450)}\bigr).
\]
Consequently, an identity holds in $\mathcal{P}(S_7,\cdot)$ if and only if it holds in both $S_{53}^{0}$ and $S_{(4,450)}$.
We therefore begin by characterizing the identities of these two ai-semirings and determining equational bases for them.
We first consider $S_{53}^{0}$. To describe its identities, we begin with the following characterization of the identities of $S_{53}$,
which is taken from~\cite{yrg}.
\begin{table}[ht]
\centering
\caption{The Cayley tables of $S_{53}^0$}
\label{tb634}
\begin{tabular}{c|cccc}
$+$ & $1$ & $2$ & $3$ & $4$\\
\hline
$1$ & $1$ & $1$ & $1$ & $1$\\
$2$ & $1$ & $2$ & $2$ & $2$\\
$3$ & $1$ & $2$ & $3$ & $3$\\
$4$ & $1$ & $2$ & $3$ & $4$
\end{tabular}
\qquad
\begin{tabular}{c|cccc}
$\cdot$ & $1$ & $2$ & $3$ & $4$\\
\hline
$1$ & $1$ & $1$ & $1$ & $4$\\
$2$ & $1$ & $1$ & $2$ & $4$\\
$3$ & $1$ & $2$ & $3$ & $4$\\
$4$ & $4$ & $4$ & $4$ & $4$
\end{tabular}
\end{table}

\begin{table}[ht]
\centering
\caption{The Cayley tables of $S_{(4,450)}$}
\label{tb450}
\begin{tabular}{c|cccc}
$+$ & $1$ & $2$ & $3$ & $4$\\
\hline
$1$ & $1$ & $1$ & $1$ & $1$\\
$2$ & $1$ & $2$ & $3$ & $4$\\
$3$ & $1$ & $3$ & $3$ & $1$\\
$4$ & $1$ & $4$ & $1$ & $4$
\end{tabular}
\qquad
\begin{tabular}{c|cccc}
$\cdot$ & $1$ & $2$ & $3$ & $4$\\
\hline
$1$ & $1$ & $2$ & $1$ & $4$\\
$2$ & $2$ & $2$ & $2$ & $2$\\
$3$ & $1$ & $2$ & $3$ & $4$\\
$4$ & $4$ & $2$ & $4$ & $2$
\end{tabular}
\end{table}

\begin{lem}\label{lem5301}
Let $\bq\preceq\bu$ be a nontrivial inequality such that
$\bu=\bu_1+\bu_2+\cdots+\bu_n$ and $\bu_i, \bq \in X^+_c$ for all $1\leq i \leq n$.
Then $\bq\preceq\bu$ holds in $S_{53}$
if and only if $L_{\geq 2}(\bu)\neq \emptyset$, $c(\bq)\subseteq c(\bu)$,
and for every $\bw\in S_2(\bq)$ there exists $\bw'\in S_2(\bu)$ such that $c(\bw')\subseteq c(\bw)$.
\end{lem}

To pass from $S_{53}$ to its zero extension $S_{53}^{0}$, we use the
following result of Wu et al.~\cite[Proposition~1.5]{wrz}, which
relates the identities of an ai-semiring to those of its zero
extension.

\begin{lem}\label{lem001}
Let $\bq\preceq \bu$ be an ai-semiring identity such that
$\bu=\bu_1+\bu_2+\cdots+\bu_n$, where $\bu_i, \bq \in X^+$, $1\leq i \leq n$.
Then $\bq\preceq \bu$ is satisfied by ${S}^0$ if and only if
$\bq\preceq D_\bq(\bu)$ holds in $S$.
\end{lem}

Combining Lemmas~\ref{lem5301} and~\ref{lem001}, we obtain the
following characterization of the identities of $S_{53}^{0}$.

\begin{cor}\label{cor530}
Let $\bq\preceq\bu$ be a nontrivial inequality such that
$\bu=\bu_1+\bu_2+\cdots+\bu_n$ and $\bu_i, \bq \in X^+_c$ for all $1\leq i \leq n$.
Then $\bq\preceq\bu$ holds in $S_{53}^{0}$ if and only if
$L_{\geq2}(D_{\bq}(\bu))\neq\emptyset$, $c(\bq)= c(D_{\bq}(\bu)\bigr)$,
and, for every $\bw\in S_2(\bq)$, there exists $\bw'\in S_2\bigl(D_{\bq}(\bu)\bigr)$ such that $c(\bw')\subseteq c(\bw)$.
\end{cor}

We next turn to the other factor $S_{(4,450)}$.
The following result completely describes the identities of $S_{(4, 450)}$.
\begin{lem}\label{lem45001}
Let $\bq\preceq\bu$ be a nontrivial inequality such that
$\bu=\bu_1+\bu_2+\cdots+\bu_n$ and $\bu_i, \bq \in X^+_c$ for all $1\leq i \leq n$.
Then $\bq\preceq\bu$ holds in $S_{(4, 450)}$
if and only if $c(\bq)=c(D_\bq(\bu))$, $\ell(\bq)\geq 2$ and one of the following conditions is satisfied:
\begin{itemize}
\item[$(1)$] $M_1(\bq)=\emptyset$;

\item[$(2)$] $M_1(\bq) \neq \emptyset$ and for every $x \in M_1(\bq)$, there exists some $\bu_j \in D_\bq(\bu)$ such that $m(x,\bu_j) = 1$.
\end{itemize}
\end{lem}
\begin{proof}
Suppose first that $\bq\preceq\bu$ holds in $S_{(4,450)}$.
The subalgebras $\{1,2,3\}$ and $\{2,4\}$ are isomorphic to $M_2^0$ and $N_2$, respectively.
Hence $c(\bq)=c(D_{\bq}(\bu))$ and $\ell(\bq)\geq2$.
Assume that $M_1(\bq) \neq \emptyset$, and that there is $x\in M_1(\bq)$ such that for any $\bu_j \in D_\bq(\bu)$, $m(x,\bu_j) \neq1$.
Define a semiring homomorphism $\varphi\colon\mathcal{P}_f(X^+)\to S_{(4,450)}$ by
\[
\varphi(y)=
\begin{cases}
4, & y=x,\\
3, & y\in c(\bq)\setminus\{x\},\\
2, & y\notin c(\bq).
\end{cases}
\]
It is easy to see that $\varphi(\bu)=3$ or $2$ and $\varphi(\bq)=4$, a contradiction.
Hence for every $x \in M_1(\bq)$, there exists some $\bu_j \in \bu$ such that $m(x,\bu_j) = 1$.

Conversely, suppose that the stated conditions hold, and let
$\varphi\colon\mathcal{P}_f(X^+)\to S_{(4,450)}$ be an arbitrary
semiring homomorphism.

If $\varphi(\bq)=2$, then
$\varphi(\bq)\leq\varphi(\bu)$, since $2$ is the least element of the additive order.
If $\varphi(\bq)=3$, then every variable in $c(\bq)$ is mapped to $3$.
Since $c(\bq)=c(D_{\bq}(\bu))$, some summand in $D_{\bq}(\bu)$ is
mapped to $3$, and hence $\varphi(\bq)\leq\varphi(\bu)$.

If $\varphi(\bq)=1$, then every variable in $c(\bq)$ is mapped to either $1$ or $3$,
and at least one variable $x\in c(\bq)$ is mapped to $1$.
Since $c(\bq)=c(D_{\bq}(\bu))$, there exists some $\bu_j\in D_{\bq}(\bu)$ such that $x\in c(\bu_j)$.
All variables occurring in $\bu_j$ are mapped to either $1$ or $3$,
and therefore $\varphi(\bu_j)=1$.
Since $1$ is the greatest element of the additive order, it follows that $\varphi(\bu)=1$.

Finally, suppose that $\varphi(\bq)=4$. It follows that there is a variable $x\in M_1(\bq)$ such that $\varphi(x)=4$,
while every other variable in $c(\bq)$ is mapped to either $1$ or $3$.
By condition~(2), there exists $\bu_j\in D_{\bq}(\bu)$ such that $m(x,\bu_j)=1$.
This forces that $\varphi(\bu_j)=4$, and hence
\[
\varphi(\bq)=4\leq\varphi(\bu).
\]
Therefore $\bq\preceq\bu$ holds in $S_{(4,450)}$.
\end{proof}

We adopt the notation introduced in \cite{yrg}. In particular,
for each integer $n\geq1$, and let $0\leq k\leq n$. Define
\[
\bu_{n, k}=\prod_{i=1}^n x_i+\left(\sum_{i=1}^kx_ix_{n+1}\right)+\left(\sum_{i=k+1}^nx_ix_{n+1}y_i\right).
\]
For convenience, we sometimes factor $x_{n+1}$ out of the two sums and write
\[
\bu_{n, k}=\prod_{i=1}^n x_i+\left(\sum_{i=1}^kx_i+\sum_{i=k+1}^nx_iy_i\right)x_{n+1}.
\]
Recall that a word is \emph{linear} if no variable occurs in it more than once.
For later use, we recall the following properties of the terms
$\bu_{n,k}$ established in \cite[Lemma~3.1]{yrg}.

\begin{lem}\label{lem26011801}
Let $n\geq 1$ be an integer, and let $0\leq k\leq n$.
\begin{itemize}
\item[(a)] Every word in $\bu_{n,k}$ is linear.

\item[(b)] The content intersection of any two distinct words in $\bu_{n,k}$ is a singleton.

\item[(c)] No distinct words in $\bu_{n,k}$ share a common subword of length $2$.

\item[(d)] $\bu_{n,k}$ has no subterm that is the square of a term.

\item[(e)] If $n\geq 2$ and $\bp$ and $\bq $ are words in $\bu_{n,k}$ with $\bp \leq\bq$, then $\bp=\bq$.

\item[(f)] $x_1x_2\cdots x_n$ is the unique additive subterm of $\bu_{n,k}$
whose content is contained in $\{x_1, x_2, \ldots, x_n\}$.

\item[(g)] If $n\geq 4$, then $x_1x_2\cdots x_n$ is the unique word in $\bu_{n, k}$ of length at least $4$.

\item[(h)]
If $n\geq 3$, then
the term $\bu_{n, n}$ contains exactly one word, $x_1x_2\cdots x_n$, whose length is at least $3$.

\item[(i)]
If $n\geq 3$, then
$\bu_{n, n}$ does not have a subterm of the form $\bt_1\bt_2+\bt_2\bt_3+\bt_3\bt_1$.

\item[(j)]
$\bu_{n, 1}$ does not have a subterm of the form $\bt_1\bt_2+\bt_2\bt_3+\bt_3\bt_1$.

\item[(k)] If $\bt$ is an additive subterm of $\bu_{n, k}$ that can be expressed as
a product of three terms, then $\bt$ must be a word.

\item[(l)] If $\bt$ is an additive subterm of $\bu_{n,k}$ that can be written as the product of two terms $\bt_1$ and $\bt_2$,
then $\bt$ is itself a word, or at least one of $\bt_1$, $\bt_2$ is a single variable.
\end{itemize}
\end{lem}
\begin{proof}
These assertions follow directly from the definition of
$\bu_{n,k}$; see \cite[Section~3]{yrg}.
\end{proof}

Let $\sigma_{n, n}$ denote the inequality $\bq_n \preceq \bu_{n, n}$, where
\[
\bq_n=\prod_{i=1}^{n+1}x_i.
\]
Write $\Omega$ for the set of all inequalities $\sigma_{n, n}$ with $n\geq 2$.
We shall use the following two results of Yue et al.~\cite{yrg}.
The first provides a sufficient condition for a commutative ai-semiring to be nonfinitely based.

\begin{thm}\label{thmNFB}
Let $S$ be a commutative ai-semiring
and $\Xi$ an equational basis for $S$ that contains an infinite subset of $\Omega$.
If $\bu_{m, m}$ is $\bt$-free for every inequality $\bs\preceq\bt$ in $\Xi\setminus \Omega$
and for every $m\geq 3$, then $S$ is nonfinitely based.
\end{thm}

Following \cite{yrg}, for each integer $n\geq2$,
let $\Theta_n$ denote the set of all terms of the form $\sum_{1\leq i<j\leq n+1}x_ix_j\bw_{ij}$,
where $c(\bw_{ij})\subseteq\{x_1, x_2,\dots,x_{n+1}\}\backslash\{x_i, x_j\}$
and each $\bw_{ij}$ is either a linear word or the empty word.
Note that some words in $\Theta_n$ may be repeated.
For every term $\bt\in\Theta_n$, each word in $\bt$ is linear and has length at least $2$,
and for any $1\leq i<j\leq n+1$, $x_ix_j$ is a subterm of $\bt$.
Observe that $\bu_{n, n}$ lies in $\Theta_n$.

For $n \geq 2$ and $\bv\in \Theta_n$,
let $\delta_{n,\bv}$ denote the inequality $\bq_n \preceq \bv$.
Then $\sigma_{n,n}$ coincides with $\delta_{n,\bv}$ for some $\bv\in \Theta_n$.

\begin{lem}\label{lemfree}
Let $n\geq 2$ and $m\geq 3$ be integers.
If $\bv$ is a term in $\Theta_n$ different from $\bu_{m, m}$, then $\bu_{m,m}$ is $\bv$-free.
In particular, if $m\neq n$, then $\bu_{m,m}$ is $\bu_{n, n}$-free.
\end{lem}

We now present an infinite equational basis for $\mathcal{P}(S_7, \cdot)$.
\begin{pro}\label{pro63401}
$\mathsf{V}(\mathcal{P}(S_7, \cdot))$
is the commutative ai-semiring variety defined by the identities
\begin{align}
&x^3 \approx x^2; \label{01}\\
&x_1^2x_2^2\cdots x_{n+1}^2 \preceq \sum_{i=1}^{n+1}x_i^2\bw_i\quad (n\geq1);\label{03}\\
&\bq_{n, k} \preceq \left(\sum_{i=1}^{k}x_i\bw_i\right)+\left(\sum_{i=1}^{k-1}x_i^2\bw_i'\right)+\left(\sum_{i=k+1}^{n+1}x_i^2\bw_i\right)
\quad (n\geq1, 1\leq k\leq n+1); \label{04}\\
&\bq_{n, k} \preceq \left(\sum_{i=1}^{k}x_i\bw_i\right)+\left(\sum_{i=1}^{\ell}x_i^2\bw_i'\right)
+\left(\sum_{\ell+1\leq i<j\leq k}x_ix_j\bw_{ij}\right)+\left(\sum_{i=k+1}^{n+1}x_i^2\bw_i\right),\label{05}
\end{align}
where,
\[
\bq_{n,k}
=
x_1x_2\cdots x_kx_{k+1}^2\cdots x_{n+1}^2.
\]
In \eqref{03}--\eqref{05}, all words $\bw_i$ and $\bw_i'$
that occur satisfy
\[
c(\bw_i),c(\bw_i')
 \subseteq
 \{x_1,\ldots,x_{n+1}\}\setminus\{x_i\},
\]
and may be empty. In \eqref{05}, each $\bw_{ij}$ satisfies
\[
c(\bw_{ij})
 \subseteq
 \{x_1,\ldots,x_{n+1}\}\setminus\{x_i,x_j\},
\]
and may also be empty.
In \eqref{05} $n\geq 1$, $1\leq k\leq n+1$ and $0\leq \ell< k-1$.
\end{pro}
\begin{proof}
It is straightforward to verify that $\mathcal{P}(S_7, \cdot)$ is isomorphic to a subdirect product $S_{53}^0$ and $S_{(4, 450)}$
via the congruences defined by the nontrivial blocks $\{\{1, 2, 3, 5\}, \{4, 7\}\}$ and $\{\{1, 4\}, \{2, 6\}, \{3, 7\}, \{5, 8\}\}$.
By Corollary~\ref{cor530} and Lemma~\ref{lem45001}, $\mathcal{P}(S_7, \cdot)$ satisfies the identities \eqref{01}--\eqref{05}.
It suffices to show that every inequality holding in $\mathcal{P}(S_7, \cdot)$ is derivable from \eqref{01}--\eqref{05}.
Consider such a nontrivial inequality $\bq\preceq \bu$, where
$\bu=\bu_1+\bu_2+\cdots+\bu_t$ and $\bu_i, \bq \in X_c^+$, $1 \leq i \leq t$.
By Corollary~\ref{cor530} and Lemma~\ref{lem45001},
$L_{\geq 2}(D_\bq(\bu))\neq \emptyset$, $c(\bq)=c(D_\bq(\bu))$, $\ell(\bq)\geq2$,
and for any $\bv\in S_2(\bq)$, there exists $\bv'\in S_2(D_\bq(\bu))$ such that $c(\bv')\subseteq c(\bv)$.

\textbf{Case 1.} $M_1(\bq)=\emptyset$. Applying identity \eqref{01}, we have
\[
\mathbf{q}=x_1^2x_2^2\cdots x_{n+1}^2.
\]
For any $\mathbf{v}\in S_2(\mathbf{q})$, there exists $\mathbf{v}'\in S_2(D_{\mathbf{q}}(\mathbf{u}))$ such that $c(\mathbf{v}')\subseteq c(\mathbf{v})$. Hence, for each $1\leq i\leq n+1$, there is $\bu_i\in D_\bq(\bu)$ such that $x_i^2\in S_2(\bu_i) $, and so
\[
\bu \succeq\bu_1+\cdots+\bu_{n+1}\stackrel{\eqref{01}}\approx x_1^2\bu_1'+\cdots+x_{n+1}^2\bu_{n+1}'\stackrel{\eqref{03}}\succeq x_1^2x_2^2\cdots x_{n+1}^2=\bq,
\]
which derives the inequality $\bu \succeq \bq$.

\textbf{Case 2.} $M_1(\bq)\neq\emptyset$. Applying identity \eqref{01}, we may write
\[
\bq=x_1x_2\cdots x_k x_{k+1}^2x_{k+2}^2\cdots x_{n+1}^2,
\]
for some $1\leq k\leq n+1$.
By Lemma~\ref{lem45001}, for each $x\in M_1(\bq)$, there exists $\bu_x\in D_{\bq}(\bu)$ such that $m(x,\bu_x)=1$.
In particular, for each $1\leq i\leq k$, choosing $x=x_i$ yields a summand $\bu_i\in D_{\bq}(\bu)$ with $m(x_i,\bu_i)=1$.
Moreover, for every $k+1\leq i\leq n+1$, since $x_i^2\in S_2(\bq)$, there exists a summand in $D_{\bq}(\bu)$ containing $x_i^2$.
Using \eqref{01}, we may write this summand as $x_i^2\bp_i$, where $c(\bp_i)\subseteq c(\bq)\setminus\{x_i\}$.

\textbf{Subcase 2.1.} For any $1\leq i<j\leq k$,
there is $\bv_{m} \in S_2(D_\bq(\bu))$ such that $\bv_{m}=x_i^2$ or $\bv_{m}=x_j^2$.
One can deduce that $\{i\mid 1\leq i\leq k, x_i^2\in S_2(D_\bq(\bu))\}$
contains at least $k-1$ elements.
We may assume that it contains $1, 2, \ldots, k-1$.
Then $x_1^2\bp_1, x_2^2\bp_2, \dots, x_{k-1}^2\bp_{k-1}\in D_\bq(\bu)$ for some $\bp_1, \bp_2, \ldots, \bp_{k-1} \in X^*$.
Since for each $x\in M_1(\bq)$, there exists $\bu_x\in D_{\bq}(\bu)$ such that $m(x,\bu_x)=1$,
it follows that there exists $x_1\bp_1',\dots, x_k\bp_k'\in D_\bq(\bu)$ with $x_i\not\in c(\bp_i')$.
Now we have
\begin{align*}
\bu
&\succeq \left(\sum_{i=1}^{k}x_i\bp_i'\right)+\left(\sum_{i=1}^{k-1}x_i^2\bp_i\right)+\left(\sum_{i=k+1}^{n+1}x_i^2\bp_i\right)\\
&\succeq x_1\cdots x_kx_{k+1}^2\cdots x_{n+1}^2, &&(\text{by}~\eqref{04})
\end{align*}
where $\bp_i, \bp'_i$ may be empty, $1\leq i \leq n+1$.
This implies $\bu \succeq \bq$.

\textbf{Subcase 2.2.} There exist $1\leq i<j\leq k$ such that
for every $\bv\in S_2(D_\bq(\bu))$, $\bv\neq x_{i}^2$ and $\bv\neq x_{j}^2$.
Consequently, $x_{i}x_{j} \in S_2(D_\bq(\bu))$.
Hence $x_{i}x_{j}\bp_{ij}\in D_\bq(\bu)$ for some $\bp_{ij}\in X^*$.

For convenience, we may assume that for some integer $\ell$ with $0\leq \ell<k-1$,
$x_i^2\in S_2(D_\bq(\bu))$ for every $1\leq i\leq \ell$, and $x_ix_j\in S_2(D_\bq(\bu))$ for every $\ell+1\leq i<j\leq k$.
Then $x_i^2\bp_i\in D_\bq(\bu)$ for some $\bp_i\in X^*$, $1\leq i\leq\ell$;
$x_ix_j\bp_{ij}\in D_\bq(\bu)$ for some $\bp_{ij}\in X^*$, $\ell+1\leq i<j\leq k$.
Combined with Lemma~\ref{lem45001}, we have that there is $\bu_j\in D_\bq(\bu)$ $(1\leq j\leq k)$ such that $m(x_j, \bu_j)=1$.
Suppose that $x_i\bp'_i \in D_\bq(\bu)$ for all $1\leq i\leq k$ and that $x_i\notin c(\bp'_i)$.
Now
\begin{align*}
\bu
&\succeq \sum_{i=1}^{k}x_i\bp'_i+\sum_{i=1}^{\ell}x_i^2\bp_i
+\sum_{\ell+1\leq i<j\leq k}x_ix_j\bp_{ij}+\sum_{i=k+1}^{n+1}x_i^2\bp_i\\
&\succeq x_1\cdots x_kx_{k+1}^2\cdots x_{n+1}^2=\bq, &&(\text{by}~\eqref{05})
\end{align*}
where $\bp'_i$, $\bp_i$ may be empty for every $1\leq i \leq n+1$,
and $\bp_{ij}$ is either empty or linear for every $\ell+1\leq i<j\leq k$.
This derives the inequality $\bu \succeq \bq$.
\end{proof}

\begin{thm}
The ai-semiring $\mathcal{P}(S_7, \cdot)$ is nonfinitely based.
\end{thm}
\begin{proof}
By Proposition~\ref{pro63401},
\[
\Xi=\{\eqref{01}, \eqref{03}, \eqref{04}, \eqref{05}\setminus\delta_{n, \bv}\}
\cup\{\delta_{n, \bv} \mid n\geq 2, \bv\in \Theta_n\}
\]
is an equational basis for the commutative ai-semiring $\mathcal{P}(S_7, \cdot)$, which contains $\Omega$.
By Theorem~\ref{thmNFB}, to show that $\mathcal{P}(S_7, \cdot)$ is nonfinitely based,
it suffices to prove that for any $m\geq 3$, $\bu_{m, m}$ is $\bw$-free for every term $\bw$ in the sets
\[
\Gamma=\{\bt \mid \bt ~\text{is an upper side of an inequality in}~\{\eqref{01}, \eqref{03}, \eqref{04}, \eqref{05}\setminus\delta_{n, \bv}\}
\]
and
\[
\{\bv \mid \bv\in \Theta_n\setminus \{\bu_{n,n}\}, n\geq 2\}.
\]
Indeed, the required freeness follows from Lemma~\ref{lem26011801}(d) for $\mathbf{w} \in\Gamma$
and from Lemma~\ref{lemfree} for $\mathbf{w}\in \{\bv \mid \bv\in \Theta_n\setminus  \{\bu_{n,n}\}, n\geq 2\}$.
\end{proof}

\section{Equational basis for $\mathcal{P}^+(S_7, \cdot)$}\label{sec4}
In this section, we give a sufficient condition for an ai-semiring to be nonfinitely based.
Applying this condition, we prove that $\mathcal{P}^+(S_7, \cdot)$ likewise has no finite equational basis.
We continue to use the notation of Section~\ref{sec3}.
In particular, for each integer $n \geq 1$,
\[
\mathbf{u}_{n, 1}=\prod_{i=1}^n x_i+x_1x_{n+1}+\sum_{i=2}^n x_iy_ix_{n+1}.
\]

\begin{pro}\label{pro01}
Let $n \geq 1$ be an integer and $1\leq k \leq n$.
Then $\mathbf{u}_{m, 1}$ is $\mathbf{u}_{n, k}$-free for every $m \geq 4$.
In particular, when $n = m$ we require $k \geq 2$.
\end{pro}
\begin{proof}
Suppose for contradiction that $\bu_{m, 1}$ is not $\bu_{n, k}$-free for some $m \geq 4$, $n\geq1$ and $1\leq k\leq n$.
Then there exist terms $\bp, \br\in P_f(X_c^+)$ and a substitution $\varphi$ such that
\begin{equation}\label{26082503}
\bp\varphi(\bu_{n,k})+\br=\bu_{m, 1},
\end{equation}
where $\bp$ may be the empty word, and $\br$ may be the empty set.

If $n>m$, then $n\geq 5$, and so every word in $\bp\varphi(x_1x_2\cdots x_n)$ has length at least $5$.
By Lemma~\ref{lem26011801}(g), $\bp\varphi(x_1x_2\cdots x_n)=x_1x_2\cdots x_m$.
Thus $n\leq m$, contradicting $n>m$.

If $n=m$, then $k \geq 2$ and $n\geq4$, and so $\bp\varphi(x_1x_2\cdots x_n)=x_1x_2\cdots x_m$.
This forces each $\varphi(x_i)$ to be a single variable and $\bp$ is the empty word.
Moreover, the variables $\varphi(x_1),\ldots,\varphi(x_n)$ are pairwise distinct.
Choose any word $\bw\in\varphi(x_{n+1})$. Since $k\geq2$, both
\[
\varphi(x_1)\bw
\quad\text{and}\quad
\varphi(x_2)\bw
\]
occur in $\bu_{m,1}$. These two words are distinct. If
$\ell(\bw)\geq2$, then they share the common subword $\bw$ of
length at least $2$, contradicting Lemma~\ref{lem26011801}(c).
If $\ell(\bw)=1$, then they are two distinct words of length $2$,
whereas $\bu_{m,1}$ contains only one word of length $2$, namely
$x_1x_{m+1}$. This is again a contradiction.

Now suppose that $n<m$.
Our first goal is to extract enough information about $\bp$, $\varphi$, $n$ and $k$ from \eqref{26082503}.
\begin{claim}\label{claim1}
$n\geq2$.
\end{claim}
\begin{proof}[Proof of Claim $\ref{claim1}$.]
If $n=1$, then $k=1$, and so $\bu_{n,1}=x_1+x_1x_2$.
Consequently,
\[
\bu_{m,1}=\bp\varphi(\bu_{1,1})+\br=\bp\varphi(x_1+x_1x_2)+\br=\bp\varphi(x_1)+\bp\varphi(x_1)\varphi(x_2)+\br.
\]
So $\bu_{m,1}$ would contain two distinct words in which one is a proper subword of the other.
This contradicts Lemma~\ref{lem26011801}(e).
Thus $n\geq2$.
\end{proof}

\begin{claim}\label{claim2}
The sets $c(\varphi(x_1)), c(\varphi(x_2)), \ldots, c(\varphi(x_n))$ are pairwise disjoint.
\end{claim}
\begin{proof}[Proof of Claim $\ref{claim2}$.]
Assume that for some distinct indices $i, j\in\{1, 2, \ldots, n\}$
the intersection $c(\varphi(x_i))\cap c(\varphi(x_j))$ is nonempty.
Then $\varphi(x_1x_2\cdots x_n)$ would have a subterm that is the square of a variable,
and subsequently $\bu_{m,1}$ would also have such a subterm,
contradicting Lemma~\ref{lem26011801}(d).
\end{proof}

\begin{claim}\label{claim3}
The term $\bp$ is a word.
\end{claim}
\begin{proof}[Proof of Claim $\ref{claim3}$.]
Suppose that $\bp_1$ and $\bp_2$ are distinct words in $\bp$.
By \eqref{26082503}, both $\bp_1\varphi(x_1x_2\cdots x_n)$ and $\bp_2\varphi(x_1x_2\cdots x_n)$ are contained in $\bu_{m,1}$.
Claim~\ref{claim1} gives $n\geq2$.
Consequently, $\bu_{m,1}$ would contain two distinct words sharing a common subword of length $2$,
which contradicts Lemma~\ref{lem26011801}$(c)$.
\end{proof}

\begin{claim}\label{claim4}
$\displaystyle \bp\varphi\left(\left(\sum_{i=1}^kx_i
               +\sum_{i=k+1}^nx_iy_i\right)x_{n+1}\right)\neq x_1x_2\cdots x_m$.
\end{claim}
\begin{proof}[Proof of Claim $\ref{claim4}$.]
Suppose for contradiction that
\[
\bp\varphi\left(\left(\sum_{i=1}^kx_i
               +\sum_{i=k+1}^nx_iy_i\right)x_{n+1}\right)= x_1x_2\cdots x_m.
\]
If $\bp$ is nonempty, then we may assume that
\[
\bp=x_1x_2\cdots x_t
\]
and
\[
\varphi\left(\left(\sum_{i=1}^kx_i
               +\sum_{i=k+1}^nx_iy_i\right)x_{n+1}\right)=x_{t+1}\cdots x_m
\]
for some $1\leq t<m$.
This implies that $x_j\leq\varphi(x_{n+1})$ for some $t+1\leq j\leq m$,
and that $c(\bp\varphi(x_1\cdots x_n))\subseteq\{x_1, x_2, \dots, x_m\}$.
By \eqref{26082503} and Lemma~\ref{lem26011801}$(f)$, one can deduce that $\bp\varphi(x_1\cdots x_n)=x_1x_2\cdots x_m$,
whence $\varphi(x_1\cdots x_n)=x_{t+1}\cdots x_m$.
Thus $x_j\leq\varphi(x_r)$ for some $1\leq r\leq n$,
and so $x_j^2\leq\varphi(x_rx_{n+1})\leq \bu_{m, 1}$,
contradicting Lemma~\ref{lem26011801}(d).

If $\bp$ is empty, then a similar argument leads to the same contradiction.
\end{proof}

\begin{claim}\label{claim5}
The word $\bp$ is either empty or a single variable.
\end{claim}
\begin{proof}[Proof of Claim $\ref{claim5}$.]
By Claim~\ref{claim3}, $\bp$ is a word.
If its length exceeds $1$, then by \eqref{26082503} and Lemma~\ref{lem26011801}(g),
\[
\bp\varphi\left(\left(\sum_{i=1}^kx_i+\sum_{i=k+1}^nx_iy_i\right)x_{n+1}\right)=x_1x_2\cdots x_m.
\]
This contradicts Claim~\ref{claim4}.
Thus the length of $\bp$ is at most $1$;
therefore $\bp$ is either empty or a single variable.
\end{proof}

%\begin{claim}\label{claim021}
%If $\bp$ is not the empty word, then $k\geq 1$.
%\end{claim}
%\begin{proof}[Proof of Claim $\ref{claim021}$.]
%Suppose that $k=0$.
%Then for every $1\leq i\leq n$, by Claim~\ref{claim011}, $\bp\varphi(x_iy_ix_{n+1})=x_1x_2\cdots x_m$. Consequently,
%\[
%\bp\varphi\left(\left(\sum_{1\leq i\leq n}x_iy_i\right)x_{n+1}\right)=x_1x_2\cdots x_m,
%\]
%which contradicts Claim~\ref{claim02}.
%Therefore, $k\geq 1$.
%\end{proof}

\begin{claim}\label{claim6}
The word $\bp$ is empty.
\end{claim}
\begin{proof}[Proof of Claim $\ref{claim6}$.]
Suppose for contradiction that $\bp$ is nonempty.
By Claim~\ref{claim5}, $\bp$ is then a single variable.
Claim~\ref{claim1} tells us that $n\geq 2$.

Observe that every word in $\bp\varphi(x_iy_ix_{n+1})$ has length at least $4$ for every $k+1\leq i\leq n$.
By Lemma~\ref{lem26011801}(g),
\begin{equation}\label{26082504}
\bp\varphi(x_iy_ix_{n+1})=x_1x_2\cdots x_m \quad (k+1\leq i\leq n).
\end{equation}

If $k\geq 2$, then by Lemma~\ref{lem26011801}(c) and Lemma~\ref{lem26011801}(k),
$\bp\varphi(x_1x_{n+1})$ and $\bp\varphi(x_2x_{n+1})$ must be the same word.
Hence $\varphi(x_1)=\varphi(x_2)$, and so the square $\varphi(x_1)\varphi(x_1)$
is a subterm of $\bu_{m, 1}$, which contradicts Lemma~\ref{lem26011801}(d).

If $k=1$, then by \eqref{26082504}, $\bp\varphi(x_iy_ix_{n+1})=x_1x_2\cdots x_m$ for every $2\leq i\leq n$.
So $\bp=x_j$ for some $1\leq j \leq m$.
Since $\bp\varphi(x_1x_{n+1})$ and $\bp\varphi(x_2y_2x_{n+1})$ share a common subword of length $2$,
it follows from Lemma~\ref{lem26011801}(c) that $\bp\varphi(x_1x_{n+1})=\bp\varphi(x_2y_2x_{n+1})$.
Consequently,
\[
\bp\varphi\left(\left(x_1+\sum_{i=2}^nx_iy_i\right)x_{n+1}\right)= x_1x_2\cdots x_m,
\]
which contradicts Claim~\ref{claim4}.
\end{proof}

Now we complete the main proof of the proposition.
By Claim~\ref{claim6}, the equality~\eqref{26082503} reduces to the form
\begin{equation}\label{26082505}
\varphi(\bu_{n,k})+\br=\bu_{m,1}.
\end{equation}
Hence, in the subset representation, $\varphi(\bu_{n,k})$ is a subset of $\bu_{m,1}$.
From Claim~\ref{claim4} we obtain that
\begin{equation}\label{26082506}
\bt=\varphi\left(\left(\sum_{i=1}^kx_i
               +\sum_{i=k+1}^nx_iy_i\right)x_{n+1}\right)
               =\varphi\left(\sum_{i=1}^kx_i
               +\sum_{i=k+1}^nx_iy_i\right)\varphi(x_{n+1})
\end{equation}
contains a word of the form $x_1x_{m+1}$ or $x_py_px_{m+1}$ for some $2\leq p \leq m$.

\textbf{Case 1.} The term $\bt$ contains at least two distinct words.
Then $\bt$ must also contain either $x_1x_2\cdots x_m$ or $x_1x_{m+1}$ or $x_jy_jx_{m+1}$ for some $j\neq p$ ($2\leq j\leq m$).

\textbf{Subcase 1.1.} The term $\bt$ contains $x_1x_2\cdots x_m$. Then
\[
\bt=x_1x_2\cdots x_m+x_py_px_{m+1}=x_p(x_1x_2\cdots x_{p-1}x_{p+1}\cdots x_m+y_px_{m+1})
\]
or
\[
\bt=x_1x_2\cdots x_m+x_1x_{m+1}=x_1(x_2\cdots x_m+x_{m+1}).
\]
If this were impossible, $\bt$ would have no nontrivial multiplicative subterm, contradicting \eqref{26082506}.

If $\bt=x_p(x_1x_2\cdots x_{p-1}x_{p+1}\cdots x_m+y_px_{m+1})$,
then by Claim~\ref{claim2}, $\varphi(x_{n+1})=x_p$ and
\begin{equation}\label{26082701}
\varphi\left(\sum_{i=1}^kx_i+\sum_{i=k+1}^nx_iy_i\right)=x_1x_2\cdots x_{p-1}x_{p+1}\cdots x_m+y_px_{m+1}.
\end{equation}
Consequently,
\[
c(\varphi(x_1x_2\cdots x_n))\subseteq \{x_1,x_2,\dots, x_{p-1}, y_p, x_{p+1},\dots, x_m, x_{m+1}\},
\]
so $x_p, y_j\notin c(\varphi(x_1x_2\cdots x_n))$ for every $j\neq p$ ($2\leq j \leq m$).
This implies that $\varphi(x_1x_2\cdots x_n)=x_1x_{m+1}$, and so $n=2$, each $\varphi(x_i)$ to be a single variable.
Combined with \eqref{26082701}, it is impossible.
Hence $\varphi(x_1x_2\cdots x_n)\nsubseteq \bu_{m,1}$, contradicting \eqref{26082505}.

If $\bt=x_1(x_2\cdots x_m+x_{m+1})$,
then by Claim~\ref{claim2}, $\varphi(x_{n+1})=x_1$ and
\begin{equation}\label{26082702}
\varphi\left(\sum_{i=1}^kx_i+\sum_{i=k+1}^nx_iy_i\right)=x_2\cdots x_m+x_{m+1}.
\end{equation}
Consequently,
\[
c(\varphi(x_1x_2\cdots x_n))\subseteq \{x_2, x_3,\dots, x_m, x_{m+1}\},
\]
so $x_1, y_j\notin c(\varphi(x_1x_2\cdots x_n))$ for every $2\leq j \leq m$.
Hence every word in $\varphi(x_1x_2\cdots x_n)$ does not belong to $\bu_{m, 1}$,
and so $\varphi(x_1x_2\cdots x_n)\nsubseteq \bu_{m,1}$, contradicting \eqref{26082505}.

\textbf{Subcase 1.2.}
The term $\bt$ contains $x_jy_jx_{m+1}$ for some $j\neq p$ ($2\leq j \leq m$).
Then $\bt$ does not contain $x_1x_2\cdots x_m$.
By \eqref{26082505},
we may write
\[
\bt=\sum_{i=2}^\ell x_iy_ix_{m+1},\quad or\quad \sum_{i=2}^\ell x_iy_ix_{m+1}+x_1x_{m+1}
\]
for some $3\leq\ell\leq m$. Since $\bt$ has the unique multiplicative factorization
\[
\bt=\left(\sum_{i=2}^\ell x_iy_i\right)x_{m+1},\quad or\quad \left(\sum_{i=2}^\ell x_iy_i+x_1\right)x_{m+1}
\]
it follows from Claim~\ref{claim2} that $\varphi(x_{n+1})=x_{m+1}$ and
\begin{equation}\label{26012901}
\varphi\left(\sum_{i=1}^kx_i+\sum_{i=k+1}^nx_iy_i\right)=\sum_{{i=2}}^\ell x_iy_i,\quad or\quad \left(\sum_{i=2}^\ell x_iy_i+x_1\right).
\end{equation}
Thus
\begin{equation}\label{26012902}
c(\varphi(x_1x_2\cdots x_n))\subseteq\{x_1, x_2, y_2,\dots, x_\ell, y_\ell\},
\end{equation}
and so $x_{m+1}\notin c(\varphi(x_1x_2\cdots x_n))$.
From \eqref{26082505} we have that $\varphi(x_1x_2\cdots x_n)=x_1\cdots x_m$.
Combining \eqref{26012901} and \eqref{26012902},
one can deduce that $\ell=m$, $k=1$ and so
\[
\varphi\left(x_1+\sum_{i=2}^nx_iy_i\right)=x_1+\sum_{{i=2}}^m x_iy_i.
\]
Therefore, $\varphi(x_i)$ is a variable for every $1\leq i\leq n$, which forces $n=m$, a contradiction.

\textbf{Subcase 1.3.}
The term $\bt$ contains $x_1x_{m+1}$. This case is included in Subcases 1.1 and 1.2.

\textbf{Case 2.} The term $\bt$ coincides with $x_py_px_{m+1}$ or $x_1x_{m+1}$.
If $\bt$ coincides with $x_py_px_{m+1}$, then
by Claim~\ref{claim2}, $\varphi(x_{n+1})$ is a single variable;
we may assume that $\varphi(x_{n+1})=x_p$.
Then
\[
\varphi\left(\sum_{1\leq i\leq k}x_i+\sum_{k+1\leq i\leq n}x_iy_i\right)=y_px_{m+1}.
\]
It follows that $c(\varphi(x_1x_2\cdots x_n))\subseteq\{y_p, x_{m+1}\}$,
and so $\varphi(x_1x_2\cdots x_n)\nsubseteq \bu_{m,1}$, contradicting \eqref{26082505}.
If $\bt$ coincides with $x_1x_{m+1}$, then
\[
\varphi\left(\sum_{1\leq i\leq k}x_i+\sum_{k+1\leq i\leq n}x_iy_i\right)=x_{m+1}\quad or\quad x_1,
\]
contradicting Claim~\ref{claim2}.
This completes the proof.
\end{proof}

Let $\sigma_{n, 1}$ denote the inequality $\bq_n \preceq \bu_{n, 1}$, where
\[
\bq_n=\prod_{i=1}^{n+1}x_i.
\]
Write $\Omega_1$ for the set of all inequalities $\sigma_{n, 1}$ with $n\geq 2$.
The following result explores a sufficient condition for an additively idempotent semiring to be nonfinitely based.

\begin{thm}\label{thm02}
Let $S$ be a commutative ai-semiring
and $\Sigma$ an equational basis for $S$ that contains an infinite subset of $\Omega_1$.
If $\bu_{m, 1}$ is $\bt$-free for every inequality $\bs\preceq\bt$ in $\Sigma\setminus \Omega_1$
and for every $m\geq 4$, then $S$ is nonfinitely based.
\end{thm}
\begin{proof}
By hypothesis and the compactness theorem of equational logic,
it is enough to prove that no finite subset of the set $\Sigma$ defines the variety $\mathsf{V}(S)$.

Let $\Sigma'$ be an arbitrary finite subset of $\Sigma$.
Since $\Sigma$ contains an infinite subset of $\Omega_1$, we can choose an integer $m$ such that
\[
m>\max\{n\geq 1 \mid \sigma_{n, 1}\in \Sigma'\}~ \text{and} ~\sigma_{m, 1}\in\Sigma.
\]
Then $\sigma_{m, 1}\notin\Sigma'$.

To prove that $\Sigma'$ cannot define the variety $\mathsf{V}(S)$,
it suffices to verify that $\sigma_{m, 1}$ is not derivable from $\Sigma'$.
By hypothesis, together with Lemma~\ref{lem26012701} and Lemma~\ref{lem02},
it is sufficient to show that $\bu_{m, 1}$ is $\bu_{n, 1}$-free for all $1 \leq n<m$,
which follows from Proposition~\ref{pro01} immediately.
This completes the proof.
\end{proof}

\begin{pro}\label{pro02}
Let $n\geq2$, and let $\bv\in\Theta_n$ be such that $\bq_n\preceq\bv$ is nontrivial.
 Then $\bu_{m, 1}$ is $\bv$-free for every $m\geq 4$.
\end{pro}
\begin{proof}
Now suppose for contradiction that $\bu_{m, 1}$ is not $\bv$-free for some $m\geq 4$, $n\geq2$ and $\bv\in \Theta_n$.
Then there exist terms $\bp, \br\in P_f(X_c^+)$ and a substitution $\varphi$ such that
\begin{equation}\label{26082501}
\bp\varphi(\bv)+\br=\bu_{m, 1},
\end{equation}
where $\bp$ may be the empty word, and $\br$ may be the empty set.
Since $\bq_n\preceq\bv$ is nontrivial, the term $\bv$ contains
at least two distinct words. Indeed, if all the words
$x_ix_j\bw_{ij}$ occurring in $\bv$ coincided with a single word
$\bw$, then $\bw$ would contain every pair $x_ix_j$. Hence
\[
c(\bw)=\{x_1,\ldots,x_{n+1}\}.
\]
Since every word occurring in $\bv$ is linear, we would have
$\bw=\bq_n$. Thus $\bq_n$ would be a summand of $\bv$, making
$\bq_n\preceq\bv$ trivial, a contradiction.

\begin{claim}\label{claimqj01}
For distinct words $x_{i_1}x_{j_1}\bw_{i_1j_1}$ and $x_{i_2}x_{j_2}\bw_{i_2j_2}$ in $\bv$,
we have
\[
\bp\varphi(x_{i_1}x_{j_1}\bw_{i_1j_1})\neq\bp\varphi(x_{i_2}x_{j_2}\bw_{i_2j_2}).
\]
\end{claim}
\begin{proof}[Proof of Claim $\ref{claimqj01}$.]
By the equality~\eqref{26082501} and Lemma~\ref{lem26011801}(e),
the sets $c(x_{i_1}x_{j_1}\bw_{i_1j_1})$ and $c(x_{i_2}x_{j_2}\bw_{i_2j_2})$ are incomparable.
Suppose for contradiction that
\[
\bp\varphi(x_{i_1}x_{j_1}\bw_{i_1j_1})=\bp\varphi(x_{i_2}x_{j_2}\bw_{i_2j_2}).
\]
Then, by Lemma~\ref{lem26011801}(l) and the cancellation property of $P_f(X^+)$ (see the preliminaries),
we obtain
\begin{equation}\label{26012301}
\varphi(x_{i_1}x_{j_1}\bw_{i_1j_1})=\varphi(x_{i_2}x_{j_2}\bw_{i_2j_2}).
\end{equation}

\textbf{Case 1.} $c(x_{i_1}x_{j_1}\bw_{i_1j_1})\cap c(x_{i_2}x_{j_2}\bw_{i_2j_2})$ is empty.
Take a variable $y$ in $\varphi(x_{i_1})$. By \eqref{26012301}, $y$ must appear in
$\varphi(x_{i_2})$, $\varphi(x_{j_2})$, or $\varphi(x_s)$ for some $x_s\leq \bw_{i_2j_2}$.
Consequently, $y^2$ is a subterm of
$\varphi(x_{i_1}x_{i_2})$, $\varphi(x_{i_1}x_{j_2})$, or $\varphi(x_{i_1}x_s)$.
Therefore, $y^2\leq \bu_{m, 1}$, which contradicts Lemma~\ref{lem26011801}(d).

\textbf{Case 2.} $c(x_{i_1}x_{j_1}\bw_{i_1j_1})\cap c(x_{i_2}x_{j_2}\bw_{i_2j_2})$ is nonempty.
Without loss of generality, we may assume that $x_{i_1}=x_{i_2}$.
By \eqref{26012301}, it follows that
\[
\varphi(x_{j_1}\bw_{i_1j_1})=\varphi(x_{j_2}\bw_{i_2j_2}).
\]

If $c(x_{j_1}\bw_{i_1j_1})\cap c(x_{j_2}\bw_{i_2j_2})$ is empty,
we are back to Case~1 and obtain a contradiction.
If the intersection remains nonempty,
we iterate the same step: after finitely many iterations
we must arrive at two words with disjoint contents whose $\varphi$-images coincide, which again reduces to Case~1.
Since the two original words are distinct and linear, after
cancelling all common variables, the two remaining words are
nonempty and have disjoint contents.
\end{proof}

\begin{claim}\label{claimqj03}
The term $\bp$ is a word.
\end{claim}
\begin{proof}[Proof of Claim $\ref{claimqj03}$.]
Suppose that $\bp_1$ and $\bp_2$ are distinct words in $\bp$.
By \eqref{26082501}, both $\bp_1\varphi(x_1x_2\bw_{12})$ and $\bp_2\varphi(x_1x_2\bw_{12})$ are contained in $\bu_{m,1}$.
Consequently, $\bu_{m,1}$ would contain two distinct words sharing a common subword of length $2$,
which contradicts Lemma~\ref{lem26011801}(c).
\end{proof}

\begin{claim}\label{claimqj04}
$\bp$ is either empty or a single variable.
\end{claim}
\begin{proof}[Proof of Claim $\ref{claimqj04}$.]
By Claim~\ref{claimqj03}, we know that $\bp$ is a word.
If $\ell(\bp)\geq 2$, then by \eqref{26082501} and Lemma~\ref{lem26011801}(g),
for distinct words $x_{i_1}x_{j_1}\bw_{i_1j_1}$ and $x_{i_2}x_{j_2}\bw_{i_2j_2}$ in $\bv$,
we have
\[
\bp\varphi(x_{i_1}x_{j_1}\bw_{i_1j_1})=\bp\varphi(x_{i_2}x_{j_2}\bw_{i_2j_2})=x_1x_2\cdots x_m.
\]
This contradicts Claim~\ref{claimqj01}.
Thus the length of $\bp$ is at most $1$;
therefore $\bp$ is either empty or a single variable.
\end{proof}

\begin{claim}\label{claimqj05}
The word $\bp$ is empty.
\end{claim}
\begin{proof}[Proof of Claim $\ref{claimqj05}$.]
Suppose for contradiction that $\bp$ is nonempty.
By Claim~\ref{claimqj04}, $\bp$ is then a single variable.
For distinct words $x_{i_1}x_{j_1}\bw_{i_1j_1}$ and $x_{i_1}x_{j_2}\bw_{i_1j_2}$ in $\bv$,
Claim~\ref{claimqj01} tells us that
\[
\bp\varphi(x_{i_1}x_{j_1}\bw_{i_1j_1})\neq\bp\varphi(x_{i_1}x_{j_2}\bw_{i_1j_2}).
\]

Observe that $\bp\varphi(x_{i_1}x_{j_1}\bw_{i_1j_1})$ and $\bp\varphi(x_{i_1}x_{j_2}\bw_{i_1j_2})$ share a common subword of length $2$,
which contradicts Lemma~\ref{lem26011801}(c).
Thus $\bp$ is empty.
\end{proof}

By Claim~\ref{claimqj05}, the equality~\eqref{26082501} reduces to the form
\begin{equation}\label{26082502}
\varphi(\bv)+\br=\bu_{m, 1}.
\end{equation}

\begin{claim}\label{claimqj02}
The term $\bv$ contains exactly one word of length at least $3$.
\end{claim}
\begin{proof}[Proof of Claim $\ref{claimqj02}$.]
By definition, every word in $\bv$ has length at least $2$.
We argue by contradiction.

\textbf{Case 1.}
All words in $\bv$ have length $2$. Then
$\bw_{ij}=\varepsilon$ for all $1\leq i<j\leq n+1$, and consequently
\[
\varphi(x_1)\varphi(x_2)+\varphi(x_1)\varphi(x_3)+\varphi(x_2)\varphi(x_3)\subseteq \bu_{m,1},
\]
which contradicts Lemma~\ref{lem26011801}(j).

\textbf{Case 2.}
$x_{i_1}x_{j_1}\bw_{i_1j_1}$ and $x_{i_2}x_{j_2}\bw_{i_2j_2}$
are distinct words in $\bv$ with length at least $3$.
Then both $\bw_{i_1j_1}$ and $\bw_{i_2j_2}$ are nonempty.
By Claim~\ref{claimqj01},
$\varphi(x_{i_1}x_{j_1}\bw_{i_1j_1})$ and $\varphi(x_{i_2}x_{j_2}\bw_{i_2j_2})$
cannot be simultaneously equal to $x_1x_2\cdots x_m$, and so
$x_1x_{m+1}$ or $x_iy_ix_{m+1}\in \varphi(x_{i_1}x_{j_1}\bw_{i_1j_1})$; or
$x_1x_{m+1}$ or $x_iy_ix_{m+1}\in \varphi(x_{i_2}x_{j_2}\bw_{i_2j_2})$ for some $1\leq i\leq m$.
We may assume that $x_1x_{m+1}$ or $x_iy_ix_{m+1}\in \varphi(x_{i_1}x_{j_1}\bw_{i_1j_1})$.
If $x_1x_{m+1}\in \varphi(x_{i_1}x_{j_1}\bw_{i_1j_1})$, then it could not be expressed as a product of three terms;
hence $x_1x_{m+1}\not\in \varphi(x_{i_1}x_{j_1}\bw_{i_1j_1})$.
Now suppose that $x_iy_ix_{m+1}\in \varphi(x_{i_1}x_{j_1}\bw_{i_1j_1})$.
If $\varphi(x_{i_1}x_{j_1}\bw_{i_1j_1})$ contained another word,
it could not be expressed as a product of three terms; hence it contains no other word.
Consequently,
$\varphi(x_{i_1}x_{j_1}\bw_{i_1j_1})=x_iy_ix_{m+1}$, and so $\bw_{i_1j_1}$ is a variable.
Suppose without loss of generality that
$\varphi(x_{i_1})=x_i$, $\varphi(x_{j_1})=y_i$, and $\varphi(\bw_{i_1j_1})=x_{m+1}$.

If $\varphi(x_{i_2}x_{j_2}\bw_{i_2j_2})=x_1x_2\cdots x_m$,
then there exists $x_s\in c(x_{i_2}x_{j_2}\bw_{i_2j_2})$ such that $\varphi(x_s)\neq x_i$.
Since $x_s\neq x_{j_1}$, it follows that $x_sx_{j_1}$ is a subterm of $\bv$.
This implies that $\varphi(x_sx_{j_1})$ is a subterm of $\bu_{m,1}$,
and $x_ry_i$ is a subterm of $\varphi(x_sx_{j_1})$ for some $2\leq r\leq m$, $r\neq i$.
Thus $x_ry_i$ is a subterm of $\bu_{m,1}$, a contradiction.

If $\varphi(x_{i_2}x_{j_2}\bw_{i_2j_2})\neq x_1x_2\cdots x_m$, then $\varphi(x_{i_2}x_{j_2}\bw_{i_2j_2})=x_jy_jx_{m+1}$
for some $2\leq j\leq m$ and $i\neq j$.
Consequently, $y_iy_j$ would be a subword of $\varphi(x_{j_1}x_{i_2})$, $\varphi(x_{j_1}x_{j_2})$, or $\varphi(x_{j_1}\bw_{i_2j_2})$.
Therefore, $y_iy_j\leq \bu_{m, 1}$, which is impossible.
\end{proof}

Now we complete the main proof of the proposition.
By  Lemma~\ref{lem26011801}(e), \eqref{26082502} and Claim~\ref{claimqj02}, we may write
\[
\bv=\prod_{i=1}^k x_i+\sum_{1\leq i\leq k, k+1\leq j\leq n+1}x_ix_j+\sum_{k+1\leq i<j\leq n+1}x_ix_j,
\]
where $3\leq k\leq n$.

If $\varphi(x_1x_2\cdots x_k)=x_1x_2\cdots x_m$, then Claim~\ref{claimqj01} tells us that
$x_1x_{m+1}$ or $x_iy_ix_{m+1}$ is a word in $\varphi(x_1x_{k+1})$ for some $2\leq i\leq m$.
If $x_iy_ix_{m+1}$ is a word in $\varphi(x_1x_{k+1})$, then $\varphi(x_1)=x_i$, and so
\[
c(\varphi(x_2\cdots x_k))\subseteq\{x_1,\cdots, x_{i-1}, x_{i+1},\cdots, x_m\}.
\]
Moreover, $y_ix_{m+1}$ is a word in $\varphi(x_{k+1})$.
It follows that there exists $j$ $(1\leq j\leq m, j\neq i)$ such that
$x_jy_ix_{m+1}\leq\varphi(x_2x_{k+1})$, and so
$x_jy_ix_{m+1} \leq \bu_{m, 1}$, a contradiction.
If $x_1x_{m+1}$ is a word in $\varphi(x_1x_{k+1})$, then $\varphi(x_1)=x_1$, and so
\[
c(\varphi(x_2\cdots x_k))\subseteq\{x_2,\cdots, x_m\}.
\]
Moreover, $x_{m+1}$ is a word in $\varphi(x_{k+1})$.
Since $m>n\geq k$ and $\varphi(x_1x_2\cdots x_k)=x_1x_2\cdots x_m$, we have that $|\varphi(x_j)\geq2|$ for some $2\leq j\leq k$,
and so $x_sx_t\leq \varphi(x_j)$ for some $2\leq s, t\leq m$.
Hence $x_sx_tx_{m+1}\leq\varphi(x_jx_{k+1})$, and so
$x_sx_tx_{m+1} \leq \bu_{m, 1}$, a contradiction.

If $\varphi(x_1x_2\cdots x_k)\neq x_1x_2\cdots x_m$,
then $\varphi(x_1x_2\cdots x_k)=x_iy_ix_{m+1}$ for some $2\leq i\leq m$, which forces $k=3$.
So $\varphi(x_1x_2x_3)=x_iy_ix_{m+1}$.
This implies that $y_i=\varphi(x_j)$ for some $1\leq j\leq 3$.
Notice that $x_jx_{n+1}$ is a word in $\bv$.
Combined with the equality~\eqref{26082502},
we obtain $\varphi(x_jx_{n+1})=y_ix_ix_{m+1}$, and so $\varphi(x_jx_{n+1})=\varphi(x_1x_2\cdots x_k)$,
which contradicts Claim~\ref{claimqj01}.

This completes the proof.
\end{proof}

For any $n\geq 2$, let $\mathscr{U}_{n, k}$ denote the set of all terms of the form
\[
\sum_{i=1}^{n+1}x_i\bw_i+\sum_{i=1}^{\ell}x_i^2\bp_i
+\sum_{\ell+1\leq i<j\leq k}x_ix_j\bp_{ij}+\sum_{i=k+1}^{n+1}x_i^2\bp_i,
\]
where \[
c\left(\sum_{i=1}^{n+1}x_i\bw_i\right)=\{x_1,\ldots,x_{n+1}\};
\]
all words $\bw_i$, $\bp_i$ and $\bp_{ij}$ may be empty; $n\geq 1$, $1\leq k\leq n+1$ and $0\leq \ell< k-1$.

Notice that $\bu_{n,1}\in\mathscr{U}_{n,n+1}$. Indeed, take
$\ell=0$, $k=n+1$, and
\[
\bw_1=x_{n+1},\qquad
\bw_{n+1}=x_1,
\]
and, for $2\leq i\leq n$, take
\[
\bw_i
=
\prod_{\substack{1\leq r\leq n\\r\neq i}}x_r.
\]
Furthermore, set
\[
\bp_{ij}
=
\prod_{\substack{1\leq r\leq n\\r\notin\{i,j\}}}x_r
\qquad
(1\leq i<j\leq n),
\]
and
\[
\bp_{1,n+1}=\varepsilon,\qquad
\bp_{i,n+1}=y_i
\quad (2\leq i\leq n).
\]
With these choices, the corresponding term in
$\mathscr{U}_{n,n+1}$ reduces, by additive idempotence, to
\[
\prod_{i=1}^{n}x_i
+
x_1x_{n+1}
+
\sum_{i=2}^{n}x_iy_ix_{n+1}
=
\bu_{n,1}.
\]
Thus $\mathscr{U}_{n,n+1}$ contains the upper side of
$\sigma_{n,1}$ for every $n\geq2$.

\begin{pro}\label{inequality005}
Let $n\geq2$, and let $\bv\in\mathscr U_{n, k}$ be the upper side
of a nontrivial inequality in \eqref{005} that does not belong
to $\Omega_1$. Then $\bu_{m,1}$ is $\bv$-free for every
$m\geq4$.
\end{pro}
\begin{proof}
If $\ell\geq1$ or $k\leq n$, then $\bv$ contains a word
involving a square. Hence Lemma~\ref{lem26011801}(d) implies
that $\bu_{m,1}$ is $\bv$-free.
We may therefore assume that $\ell=0$ and $k=n+1$.
Then $\mathscr{U}_{n, k}$ denote the set of all terms of the form
\[
\sum_{i=1}^{n+1}x_i\bw_i+\sum_{1\leq i<j\leq n+1}x_ix_j\bp_{ij}.
\]
In this case , we may assume that $x_i\bw_i$ and $x_ix_j\bp_{ij}$ are linear.
If one of the words $x_i\bw_i$ or $x_ix_j\bp_{ij}$ is
nonlinear, then it contains a square of some variable.
Lemma~\ref{lem26011801}(d) then immediately implies that
$\bu_{m,1}$ is $\bv$-free. Hence we may assume that all
words occurring in $\bv$ are linear.

\textbf{Case 1.} Suppose that $c(\bp_{ij})\subseteq\{x_1,\ldots,x_{n+1}\}$ for all $1\leq i<j\leq n+1$.
Then $\bv$ contains an additive subterm $\bv_0\in\Theta_n$ such that $\bq_n\preceq\bv_0$ is nontrivial.
By Proposition~\ref{pro02}, $\bu_{m,1}$ is $\bv_0$-free, and hence it is also $\bv$-free.

\textbf{Case 2.} There exist $1\leq i<j\leq n+1$ such that $c(\bp_{ij})\not\subseteq \{x_1, x_2, \dots, x_{n+1}\}$.
This implies that $c(\bp_{i'j'})\not\subseteq \{x_1, x_2, \dots, x_{n+1}\}$ for some $1\leq i'<j'\leq n+1$,
and so $x_{i'}x_{j'}\bp_{i'j'}\not\in D_{\bq_n}(\bv)$.

Now suppose for contradiction that $\bu_{m, 1}$ is not $\bv$-free for some $m\geq 4$, $n\geq2$ and $\bv\in \mathscr{U}_{n, k}$.
Then there exist terms $\bp, \br\in P_f(X_c^+)$ and a substitution $\varphi$ such that
\begin{equation}\label{26082301}
\bp\varphi(\bv)+\br=\bu_{m, 1},
\end{equation}
where $\bp$ may be the empty word, and $\br$ may be the empty set.
Combining this with the explicit form of $\mathscr{U}_{n, k}$ and identity \eqref{26082301},
we can obtain $\bw_i$ is not the empty word for any $1\leq i\leq n+1$.
So we may write
\[
\sum_{i=1}^{n+1}x_i\bw_i=x_{i_1}x_{j_1}\bw_{i_1j_1}+\cdots+x_{i_r}x_{j_r}\bw_{i_rj_r},
\]
where $\bw_{i_tj_t}$ may be the empty word for all $1\leq t\leq r$.
Since $x_ix_j\bp_{ij}$ occurs in $\bv$ for every $i<j$,
we have
\[
\{x_ix_j\mid 1\leq i<j\leq n+1\}\subseteq S_2(\bv).
\]
Together with $c(\bp_{i'j'})\not\subseteq \{x_1, x_2, \dots, x_{n+1}\}$, it follows that $x_{i'}x_{j'}\bp_{i'j'}\not\in D_{\bq_n}(\bv)$.

\begin{claim}\label{claim01}
For distinct words $x_{i_1}x_{j_1}\bw_{i_1j_1}$ and $x_{i_2}x_{j_2}\bw_{i_2j_2}$ in $\sum_{i=1}^{n+1}x_i\bw_i$,
\[
\bp\varphi(x_{i_1}x_{j_1}\bw_{i_1j_1})\neq\bp\varphi(x_{i_2}x_{j_2}\bw_{i_2j_2}).
\]
\end{claim}
\begin{proof}[Proof of Claim $\ref{claim01}$.]
The argument is identical to the proof of Claim~\ref{claimqj01}.
%By the equality~\eqref{26082301} and Lemma~\ref{lem26011801}(e),
%the sets $c(x_{i_1}x_{j_1}\bw_{i_1j_1})$ and $c(x_{i_2}x_{j_2}\bw_{i_2j_2})$ are incomparable.
%Suppose for contradiction that
%\[
%\bp\varphi(x_{i_1}x_{j_1}\bw_{i_1j_1})=\bp\varphi(x_{i_2}x_{j_2}\bw_{i_2j_2}).
%\]
%Then, by Lemma~\ref{lem26011801}(l) and the cancellation property of $P_f(X^+)$ (see the preliminaries),
%we obtain
%\begin{equation}\label{26082302}
%\varphi(x_{i_1}x_{j_1}\bw_{i_1j_1})=\varphi(x_{i_2}x_{j_2}\bw_{i_2j_2}).
%\end{equation}
%
%\textbf{Case 1.} $c(x_{i_1}x_{j_1}\bw_{i_1j_1})\cap c(x_{i_2}x_{j_2}\bw_{i_2j_2})$ is empty.
%Take a variable $y$ in $\varphi(x_{i_1})$. By \eqref{26012301}, $y$ must appear in
%$\varphi(x_{i_2})$, $\varphi(x_{j_2})$, or $\varphi(x_s)$ for some $x_s\leq \bw_{i_2j_2}$.
%Consequently, $y^2$ is a subterm of
%$\varphi(x_{i_1}x_{i_2})$, $\varphi(x_{i_1}x_{j_2})$, or $\varphi(x_{i_1}x_s)$.
%Therefore, $y^2\leq \bu_{m, 1}$, which contradicts Lemma~\ref{lem26011801}(d).
%
%\textbf{Case 2.} $c(x_{i_1}x_{j_1}\bw_{i_1j_1})\cap c(x_{i_2}x_{j_2}\bw_{i_2j_2})$ is nonempty.
%Without loss of generality, we may assume that $x_{i_1}=x_{i_2}$.
%By \eqref{26012301}, it follows that
%\[
%\varphi(x_{j_1}\bw_{i_1j_1})=\varphi(x_{j_2}\bw_{i_2j_2}).
%\]
%
%If $c(x_{j_1}\bw_{i_1j_1})\cap c(x_{j_2}\bw_{i_2j_2})$ is empty,
%we are back to Case~1 and obtain a contradiction.
%If the intersection remains nonempty,
%we iterate the same step: after finitely many iterations
%we must arrive at two words with disjoint contents whose $\varphi$-images coincide, which again reduces to Case~1.
\end{proof}

\begin{claim}\label{claim02}
The term $\bp$ is a word.
\end{claim}
\begin{proof}[Proof of Claim $\ref{claim02}$.]
Suppose that $\bp_1$ and $\bp_2$ are distinct words in $\bp$.
By \eqref{26082301}, both $\bp_1\varphi(x_1x_2\bw_{12})$ and $\bp_2\varphi(x_1x_2\bw_{12})$ are contained in $\bu_{m, 1}$.
Consequently, $\bu_{m,1}$ would contain two distinct words sharing a common subword of length $2$,
which contradicts Lemma~\ref{lem26011801}(c).
\end{proof}

\begin{claim}\label{claim03}
$\bp$ is either empty or a single variable.
\end{claim}
\begin{proof}[Proof of Claim $\ref{claim03}$.]
By Claim~\ref{claim02}, we know that $\bp$ is a word.
If $\ell(\bp)\geq 2$, then by \eqref{26082301} and Lemma~\ref{lem26011801}(g),
for distinct words $x_{i_1}x_{j_1}\bw_{i_1j_1}$ and $x_{i_2}x_{j_2}\bw_{i_2j_2}$ in $\sum_{i=1}^{n+1}x_i\bw_i$,
we have
\[
\bp\varphi(x_{i_1}x_{j_1}\bw_{i_1j_1})=\bp\varphi(x_{i_2}x_{j_2}\bw_{i_2j_2})=x_1x_2\cdots x_m.
\]
This contradicts Claim~\ref{claim01}.
Thus the length of $\bp$ is at most $1$;
therefore $\bp$ is either empty or a single variable.
\end{proof}

\begin{claim}\label{claim04}
The word $\bp$ is empty.
\end{claim}
\begin{proof}[Proof of Claim $\ref{claim04}$.]
Suppose for contradiction that $\bp$ is nonempty.
By Claim~\ref{claim03}, $\bp$ is then a single variable.
For distinct words $x_{i_1}x_{j_1}\bw_{i_1j_1}$ and $x_{i_2}x_{j_2}\bw_{i_1j_2}$ in $\sum_{i=1}^{n+1}x_i\bw_i$,
Claim~\ref{claim01} tells us that
\[
\bp\varphi(x_{i_1}x_{j_1}\bw_{i_1j_1})\neq\bp\varphi(x_{i_2}x_{j_2}\bw_{i_1j_2}).
\]
By Lemma~\ref{lem26011801}(c), $\bp\varphi(x_{i_1}x_{j_1}\bw_{i_1j_1})$ and $\bp\varphi(x_{i_2}x_{j_2}\bw_{i_1j_2})$
do not share a common subword of length $2$,
and so the intersection $c(x_{i_1}x_{j_1}\bw_{i_1j_1})$ and $c(x_{i_2}x_{j_2}\bw_{i_1j_2})$ is empty.

Since
\[
\bp\varphi(x_{i_1}x_{j_1}\bw_{i_1j_1})\neq\bp\varphi(x_{i_2}x_{j_2}\bw_{i_1j_2}),
\]
it follows that $x_qy_qx_{m+1}\in \bp\varphi(x_{i_1}x_{j_1}\bw_{i_1j_1})$ or $\bp\varphi(x_{i_2}x_{j_2}\bw_{i_1j_2})$ for some $2\leq q\leq m$.
We may assume that $x_qy_qx_{m+1}\in \bp\varphi(x_{i_1}x_{j_1}\bw_{i_1j_1})$.
By Lemma~\ref{lem26011801}(k), we have that $\bp\varphi(x_{i_1}x_{j_1}\bw_{i_1j_1})=x_qy_qx_{m+1}$, and so $\bw_{i_1j_1}$ is empty.

If $\bp\varphi(x_{i_2}x_{j_2}\bw_{i_2j_2})=x_1x_2\cdots x_m$, then $\bp=x_q$,
and so $x_ry_q$ is a subterm of $\varphi(x_{i_1}x_{i_2})$ or $\varphi(x_{j_1}x_{i_2})$ for some $2\leq r\leq m(r\neq q)$.
This implies that $x_ry_q$ is a subterm of $\bu_{m, 1}$, contradiction.

If $\bp\varphi(x_{i_2}x_{j_2}\bw_{i_2j_2})=x_py_px_{m+1}$ for some $2\leq p\leq m(p\neq q)$, then $\bp=x_{m+1}$,
and so $y_qy_p$ is a subterm of $\varphi(x_{i_1}x_{i_2})$ or $\varphi(x_{i_1}x_{j_2})$ or $\varphi(x_{j_1}x_{i_2})$ or $\varphi(x_{j_1}x_{j_2})$.
This implies that $y_qy_p$ is a subterm of $\bu_{m, 1}$, contradiction.

Thus $\bp$ is empty.
\end{proof}

By Claim~\ref{claim04}, the equality~\eqref{26082301} reduces to the form
\begin{equation}\label{26082401}
\varphi(\bv)+\br=\bu_{m, 1},
\end{equation}
where $\br$ may be the empty set.

%\begin{claim}\label{claim05}
%The term $\sum_{i=1}^{n+1}x_i\bw_i$ contains at most one word of length at least $4$.
%\end{claim}
%\begin{proof}[Proof of Claim $\ref{claim05}$.]
%$x_{i_1}x_{j_1}\bw_{i_1j_1}$ and $x_{i_2}x_{j_2}\bw_{i_2j_2}$
%are distinct words in $\sum_{i=1}^{n+1}x_i\bw_i$ with length at least $4$.
%Then both $\bw_{i_1j_1}$ and $\bw_{i_2j_2}$ are nonempty.
%By identity \eqref{26082401}, we have
%\[
%\varphi(x_{i_1}x_{j_1}\bw_{i_1j_1})=\varphi(x_{i_2}x_{j_2}\bw_{i_2j_2})=x_1x_2\cdots x_m.
%\]
%This contradicts Claim~\ref{claim01}.
%\end{proof}

\begin{claim}\label{claim06}
The term $\sum_{i=1}^{n+1}x_i\bw_i$ contains at most one word of length at least $3$.
\end{claim}
\begin{proof}[Proof of Claim $\ref{claim06}$.]
$x_{i_1}x_{j_1}\bw_{i_1j_1}$ and $x_{i_2}x_{j_2}\bw_{i_2j_2}$
are distinct words in $\sum_{i=1}^{n+1}x_i\bw_i$ with length at least $3$.
Then both $\bw_{i_1j_1}$ and $\bw_{i_2j_2}$ are nonempty.
By Claim~\ref{claim01}, we have
\[
\varphi(x_{i_1}x_{j_1}\bw_{i_1j_1})\neq\varphi(x_{i_2}x_{j_2}\bw_{i_2j_2}).
\]
Together with identity \eqref{26082401}, we may assume that
\[
\varphi(x_{i_1}x_{j_1}\bw_{i_1j_1})=x_py_px_{m+1}, \quad\varphi(x_{i_2}x_{j_2}\bw_{i_2j_2})=x_1\cdots x_m
\]
or
\[
\varphi(x_{i_1}x_{j_1}\bw_{i_1j_1})=x_py_px_{m+1}, \quad\varphi(x_{i_2}x_{j_2}\bw_{i_2j_2})=x_qy_qx_{m+1},
\]
where $2\leq p, q\leq m$ and $p\neq q$.
Since $\{x_ix_j\mid 1\leq i<j\leq n+1\}\subseteq S_2(\bv)$,
we have that $x_ry_p$ $(r\neq p)$ or $y_qy_p$ is a subterm of $\bu_{m, 1}$, a contradiction.
\end{proof}

\begin{claim}\label{claim07}
The term $\sum_{i=1}^{n+1}x_i\bw_i$ contains exactly one word of length at least $3$.
\end{claim}
\begin{proof}[Proof of Claim $\ref{claim07}$.]
By Claim~\ref{claim06}, $\sum_{i=1}^{n+1}x_i\bw_i$ contains at most one word of length at least $3$.
Suppose, to the contrary, that $\sum_{i=1}^{n+1}x_i\bw_i$ contains no word of length at least $3$.
Since each $\bw_i$ is nonempty, every word occurring in $\sum_{i=1}^{n+1}x_i\bw_i$ has length exactly $2$.
We distinguish the following possibilities according to the words $x_ix_j\bp_{ij}$.

\textbf{Case 1.}
All the words $x_ix_j\bp_{ij}$ have length $2$.
Then $\bp_{ij}=\varepsilon$ for all $1\leq i<j\leq n+1$, and consequently
\[
\varphi(x_1)\varphi(x_2)+\varphi(x_1)\varphi(x_3)+\varphi(x_2)\varphi(x_3)\subseteq \bu_{m,1},
\]
which contradicts Lemma~\ref{lem26011801}(j).

\textbf{Case 2.}
Some word $x_ix_j\bp_{ij}$ has length at least $3$.
Then
\[
x_1x_2+x_1x_3+x_2x_3y_2\subseteq \bv,
\]
together with Proposition~\ref{pro01}, $\mathbf{u}_{m, 1}$ is $\mathbf{u}_{2, 1}$-free for every $m \geq 4$,
and so $\mathbf{u}_{m, 1}$ is $\mathbf{v}$-free;
or
\[
x_1x_2+x_3x_4\subseteq \bv.
\]
%We shall show that $\varphi(x_1x_3)$ and $\varphi(x_1x_4)$ cannot be simultaneously equal to additive subterms of $\bu_{m,1}$.
%Suppose that it is not true.
By identity \eqref{26082401}, we have
\[
\varphi(x_1x_2), \varphi(x_3x_4)\subseteq \bu_{m,1},
\]
and so $\varphi(x_1x_2)$ and $\varphi(x_3x_4)$ are distinct additive subterms of $\bu_{m,1}$.
Since $c(\varphi(x_1x_2))$ and $c(\varphi(x_3x_4))$ have nonempty intersection,
we have that there is $y\in c(\bu_{m,1})$ such that $y^2\leq\varphi(x_1x_3)$ or $\varphi(x_1x_4)$,
and so $y^2\leq\bu_{m,1}$, a contradiction.
%Hence $\varphi(x_1)\varphi(x_3+x_4)$ is a subterm of $\bu_{m,1}$.
%This implies that $\varphi(x_1)=x_1$ or $x_{m+1}$ or $x_i(i\geq2)$.
%If $\varphi(x_1)=x_1$, then $\varphi(x_3+x_4)=x_{m+1}+x_2\cdots x_m$,
%and so $\varphi(x_3x_4)=x_{m+1}x_2\cdots x_m\leq\bu_{m,1}$, a contradicton.
%If $\varphi(x_1)=x_{m+1}$, then $\varphi(x_3+x_4)=x_1+\sum_{i=2}^\ell x_iy_i$ or $\sum_{i=2}^\ell x_iy_i$ for some $2\leq\ell\leq m$,
%and so $\sum_{i=2}^\ell x_iy_i\leq\varphi(x_3x_4)\leq\bu_{m,1}$, a contradicton.
%If $\varphi(x_1)=x_i(i\geq2)$, then $\varphi(x_3+x_4)=y_ix_{m+1}+x_1\cdots x_{i-1}x_{i+1}\cdots x_m$,
%and so $\varphi(x_3x_4)=y_ix_{m+1}x_1\cdots x_{i-1}x_{i+1}\cdots x_m\leq\bu_{m,1}$, a contradicton.????
Thus $\sum_{i=1}^{n+1}x_i\bw_i$ contains exactly one word of length at least $3$.
\end{proof}

Now we complete the main proof of the proposition.
By Claim~\ref{claim07}, it follows that $\sum_{i=1}^{n+1}x_i\bw_i$ contains exactly one word of length at least $3$.
Let
\[
\bv_1=\sum_{i=1}^{n+1}x_i\bw_i.
\]
By Claim~\ref{claim07}, $\bv_1$ contains a unique word of
length at least $3$. After renaming the variables, we may
assume that this word is $x_1x_2\cdots x_k$.
Since $\bq_n\preceq\bv$ is nontrivial, we have $k\leq n$.
Moreover,
\[
c(\bv_1)=\{x_1,\ldots,x_{n+1}\}.
\]
Hence some word of length $2$ in $\bv_1$ contains a variable
$x_r$ with $r>k$. We may write this word as $x_jx_r$.
Thus either $j\leq k$, or both $j,r>k$.

%Combined with $c(D_{\bq_n}(\bv))=c(\sum_{i=1}^{n+1}x_i\bw_i)$, we may write
%\[
%\sum_{i=1}^{n+1}x_i\bw_i=\prod_{i=1}^k x_i+x_jx_r+\cdots,
%\]
%where $3\leq k\leq n$, $1\leq j\leq k$ and $k+1\leq r\leq n+1$; or
%\[
%\sum_{i=1}^{n+1}x_i\bw_i=\prod_{i=1}^k x_i+x_jx_r+\cdots,
%\]
%where $3\leq k\leq n$ and $k+1\leq j,r\leq n+1$.
If $j,r>k$ (in $\sum_{i=1}^{n+1}x_i\bw_i$, the intersection of the capacities of two terms is empty),
then by the proof of Claim~\ref{claim07}, there is $y\in c(\bu_{m,1})$ such that $y^2\leq\bu_{m,1}$, a contradiction.
If $j\leq k$ (the capacity intersections of any two terms in $\sum_{i=1}^{n+1}x_i\bw_i$ are nonempty),
by Proposition~\ref{pro01} $\mathbf{u}_{m, 1}$ is $\mathbf{u}_{k, \ell}$-free for every $m \geq 4$ and $\ell\geq1$,
and so $\mathbf{u}_{m, 1}$ is $\mathbf{v}$-free.
Therefore, $\bu_{m, 1}$ is $\bv$-free for every $m\geq 4$, $n\geq2$ and $\bv\in \mathscr{U}_{n, k}$.
This completes the proof.
\end{proof}

\begin{pro}\label{pro545}
$\mathsf V(\mathcal{P}^+(S_7, \cdot)) = \mathsf V(S_{53}, S_{(4, 450)})$.
\end{pro}
\begin{proof}
It is a routine matter to verify that  $\mathcal{P}^+(S_7, \cdot)$ is isomorphic to a subdirect product $S_{(4, 545)}$ (its Cayley tables are given in Table~\ref{tb545})
and $S_{(4, 450)}$ via the congruences defined by the nontrivial blocks $\{1, 2, 3, 5\}$ and $\{\{1, 4\}, \{2, 6\}, \{3, 7\}\}$.
So $\mathsf V(\mathcal{P}^+(S_7, \cdot)) = \mathsf V(S_{(4, 545)}, S_{(4, 450)})$.
Combined with~\cite[Proposition 3.4]{yrg}, we have that $\mathsf V(S_{(4, 545)}) = \mathsf V(S_{53}, D_2)$, and so
 $\mathsf V(\mathcal{P}^+(S_7, \cdot)) = \mathsf V(S_{53}, D_2, S_{(4, 450)})$.
 Since $D_2$ can be embedded into $S_{(4, 450)}$, we therefore have
\[
\mathsf V(\mathcal{P}^+(S_7, \cdot)) = \mathsf V(S_{53}, S_{(4, 450)}).
\]
This completes the proof.
\end{proof}

By~\cite[Corollary~3.8]{yrg}, the variety $\mathsf{V}(S_{(4, 545)})$ contains $S_7$; consequently, we have the following remark.
\begin{remark}
$\mathsf V(\mathcal{P}^+(S_7, \cdot))$ contains $S_7$.
\end{remark}

\begin{table}[ht]
\caption{The Cayley tables of $S_{(4, 545)}$} \label{tb545}
\begin{tabular}{c|cccc}
$+$ & $1$ & $2$ & $3$ & $4$\\
\hline
$1$ & $1$ & $1$ & $1$ & $1$\\
$2$ & $1$ & $2$ & $2$ & $2$\\
$3$ & $1$ & $2$ & $3$ & $3$\\
$4$ & $1$ & $2$ & $3$ & $4$
\end{tabular}
\qquad
\begin{tabular}{c|cccc}
$\cdot$ & $1$ & $2$ & $3$ & $4$\\
\hline
$1$ & $1$ & $1$ & $1$ & $1$\\
$2$ & $1$ & $1$ & $2$ & $2$\\
$3$ & $1$ & $2$ & $3$ & $4$\\
$4$ & $1$ & $2$ & $4$ & $4$
\end{tabular}
\end{table}

We now present an infinite equational basis for $\mathcal{P}^+(S_7, \cdot)$.
\begin{pro}\label{pro54501}
$\mathsf{V}(\mathcal{P}^+(S_7, \cdot))$
is the commutative ai-semiring variety defined by the identities
\begin{align}
&x^3 \approx x^2; \label{001}\\
&x_1^2x_2^2\cdots x_{n+1}^2 \preceq \sum_{i=1}^{n+1}x_i^2\bw_i\quad (n\geq1);\label{003}\\
&x_1^2x_2^2\cdots x_{n+1}^2 \preceq \sum_{i=1}^{n+1}x_i\bw_i+\sum_{i=1}^{n+1}x_i^2\bp_i\quad (n\geq1); \label{004}\\
&\bq_{n,k} \preceq \sum_{i=1}^{n+1}x_i\bw_i+\sum_{i=1}^{k-1}x_i^2\bp_i+\sum_{i=k+1}^{n+1}x_i^2\bp_i \label{006}\quad (n\geq1, 1\leq k\leq n+1);\\
&\bq_{n,k} \preceq \sum_{i=1}^{n+1}x_i\bw_i+\sum_{i=1}^{\ell}x_i^2\bp_i
+\sum_{\ell+1\leq i<j\leq k}x_ix_j\bp_{ij}+\sum_{i=k+1}^{n+1}x_i^2\bp_i,\label{005}
\end{align}
where,
\[
\bq_{n,k}
=
x_1x_2\cdots x_kx_{k+1}^2\cdots x_{n+1}^2.
\]
In \eqref{003} and \eqref{004}, all words $\bw_i$ that occur satisfy
\[
c(\bw_i)\subseteq\{x_1,\ldots,x_{n+1}\}
\]
and may be empty.
In \eqref{006} and \eqref{005},
\[
c(\bq_{n,k})=c\left(\sum_{i=1}^{n+1}x_i\bw_i\right);
\]
all words $\bw_i$ satisfy
\[
c(\bw_i)
\subseteq
\{x_1,\ldots,x_{n+1}\}\setminus\{x_i\}
\quad (1\leq i\leq k),
\]
and
\[
c(\bw_i)
\subseteq
\{x_1,\ldots,x_{n+1}\}
\quad (k+1\leq i\leq n+1)
\]
and may be empty.
In \eqref{003}--\eqref{005}, all words $\bp_i$ and $\bp_{ij}$ may be empty.
In \eqref{005} $n\geq 1$, $1\leq k\leq n+1$ and $0\leq \ell< k-1$.
\end{pro}
\begin{proof}
It is a routine matter to verify that $\mathcal{P}^+(S_7, \cdot)$ satisfies the identities \eqref{001}.
By Proposition~\ref{pro545} together with Lemmas~\ref{lem5301}, and \ref{lem45001},
$\mathcal{P}^+(S_7, \cdot)$ satisfies the inequalities in \eqref{003}--\eqref{005}.
It suffices to show that every inequality holding in $\mathcal{P}^+(S_7, \cdot)$ is derivable from \eqref{001}--\eqref{005}.
Consider such a nontrivial inequality $\bq\preceq \bu$, where
$\bu=\bu_1+\bu_2+\cdots+\bu_t$ and $\bu_i, \bq \in X_c^+$, $1 \leq i \leq t$.
By Lemmas \ref{lem5301}, and \ref{lem45001},
$L_{\geq2}(\bu)\neq \emptyset$, $c(\bq)=c(D_\bq(\bu))$, $\ell(\bq)\geq2$,
and for any $\bv\in S_2(\bq)$, there exists $\bv'\in S_2(\bu)$ such that $c(\bv')\subseteq c(\bv)$.

\textbf{Case 1.} $M_1(\bq)=\emptyset$. Applying identity \eqref{001}, we have
\[
\mathbf{q}=x_1^2x_2^2\cdots x_{n+1}^2.
\]
For any $\mathbf{v}\in S_2(\mathbf{q})$, there exists $\mathbf{v}'\in S_2(\bu)$ such that $c(\mathbf{v}')\subseteq c(\mathbf{v})$.
Hence, for each $1\leq i\leq n+1$, there is a summand $\bu_i\in \bu$ such that $x_i^2\in S_2(\bu_i)$.
Since $c(\bq)=c(D_{\bq}(\bu))$, for each $1\leq i\leq n+1$, there exist summands $x_i\bp_i\in D_{\bq}(\bu)$ such that
\[
c(\bq)=c\left(\sum_{i=1}^{n+1}x_i\bp_i\right).
\]
Therefore,
\begin{align*}
\bu
&\succeq \bu_1+\cdots+\bu_{n+1}+\left(\sum_{i=1}^{n+1}x_i\bp_i\right)\\
&\approx \left(\sum_{i=1}^{n+1}x_i^2\bu'_i\right)+\left(\sum_{i=1}^{n+1}x_i\bp_i\right) &&(\text{by}~\eqref{001})\\
&\succeq x_1^2x_2^2\cdots x_{n+1}^2=\bq, &&(\text{by}~\eqref{003}, \eqref{004})
\end{align*}
where $\bu'_i$, $\bp_i$ may be empty for every $1\leq i \leq n+1$,
which derives the inequality $\bu \succeq \bq$.

\textbf{Case 2.} $M_1(\bq)\neq\emptyset$. Applying identity \eqref{001}, we may write
\[
\bq=x_1x_2\cdots x_k x_{k+1}^2x_{k+2}^2\cdots x_{n+1}^2
\]
for some $1\leq k\leq n+1$.
By Lemma~\ref{lem45001}, for each $x\in M_1(\bq)$, there exists $\bu_x\in D_{\bq}(\bu)$ such that $m(x,\bu_x)=1$.
In particular, for each $1\leq i\leq k$, choosing $x=x_i$ yields a summand $\bu_i\in D_{\bq}(\bu)$ with $m(x_i,\bu_i)=1$.
For every $k+1\leq i\leq n+1$, since $x_i^2\in S_2(\bq)$, there exists a summand of $\bu$ containing $x_i^2$.
Using identity \eqref{001}, we may write this summand as $x_i^2\bp_i$ with $x_i\notin c(\bp_i)$.

\textbf{Subcase 2.1.} For any $1\leq i<j\leq k$,
there is $\bv_{m} \in S_2(\bu)$ such that $\bv_{m}=x_i^2$ or $\bv_{m}=x_j^2$.
One can deduce that $\{i\mid 1\leq i\leq k, x_i^2\in S_2(\bu)\}$
contains at least $k-1$ elements.
We may assume that it contains $1, 2, \ldots, k-1$.
Then $x_1^2\bp_1, x_2^2\bp_2, \dots, x_{k-1}^2\bp_{k-1}\in \bu$ for some $\bp_1, \bp_2, \ldots, \bp_{k-1} \in X^*$.
Since for each $x\in M_1(\bq)$, there exists $\bu_x\in D_{\bq}(\bu)$ such that $m(x,\bu_x)=1$,
it follows that there exists $x_1\bw_1,\dots, x_k\bw_k\in D_{\bq}(\bu)$ for some $\bw_1,\dots,\bw_k\in X^*$ and $x_i\notin c(\bw_i)$.
Combined with $c(\bq)=c(D_{\bq}(\bu))$, for each $k+1\leq i\leq n+1$, there exist summands $x_i\bw_i\in D_\bq(\bu)$ such that
\[
c(\bq)=c\left(\sum_{i=1}^{k}x_i\bw_i+\sum_{i=k+1}^{n+1}x_i\bw_i\right).
\]
Consequently,
\begin{align*}
\bu
&\succeq \left(\sum_{i=1}^{k}x_i\bw_i+\sum_{i=k+1}^{n+1}x_i\bw_i\right)+\left(\sum_{i=1}^{k-1}x_i^2\bp_i\right)+\left(\sum_{i=k+1}^{n+1}x_i^2\bp_i\right)\\
&\succeq x_1\cdots x_kx_{k+1}^2\cdots x_{n+1}^2, &&(\text{by}~\eqref{006})
\end{align*}
where $\bw_i, \bp_i, \bp'_i$ may be empty, $1\leq i \leq n+1$.
This yields $\bu \succeq \bq$.

\textbf{Subcase 2.2.} There exist $1\leq i<j\leq k$ such that
for every $\bv\in S_2(\bu)$, $\bv\neq x_{i}^2$ and $\bv\neq x_{j}^2$.
Consequently, $x_{i}x_{j} \in S_2(\bu)$.
Hence $x_{i}x_{j}\bp_{ij}\in \bu$ for some $\bp_{ij}\in X^*$.

For convenience, assume that for some integer $\ell$ with $0\leq \ell<k-1$,
$x_i^2\in S_2(\bu)$ for every $1\leq i\leq \ell$, and $x_ix_j\in S_2(\bu)$ for every $\ell+1\leq i<j\leq k$.
Then $x_i^2\bp_i\in \bu$ for some $\bp_i\in X^*$, $1\leq i\leq \ell$;
$x_ix_j\bp_{ij}\in \bu$ for some $\bp_{ij}\in X^*$, $\ell+1\leq i<j\leq k$.

By Lemma~\ref{lem45001}, for each $1\leq j\leq k$, there is a summand $\bu_j\in D_{\bq}(\bu)$ such that $m(x_j,\bu_j)=1$,
and $c(\bq)=c(D_{\bq}(\bu))$.
Furthermore, there exist $x_i\bw_i\in D_{\bq}(\bu)$ for all $1\leq i\leq k$ with $x_i\notin c(\bw_i)$,
and there exist $x_i\bw_i\in D_{\bq}(\bu)$ for all $k+1\leq i\leq n+1$ such that
\[
c(\bq)=c(\sum_{i=1}^{k}x_i\bw_i+\sum_{i=k+1}^{n+1}x_i\bw_i).
\]
Therefore,
\begin{align*}
\bu
&\succeq \sum_{i=1}^{k}x_i\bw_i+\sum_{i=k+1}^{n+1}x_i\bw_i+\sum_{i=1}^{\ell}x_i^2\bp_i
+\sum_{\ell+1\leq i<j\leq k}x_ix_j\bp_{ij}+\sum_{i=k+1}^{n+1}x_i^2\bp_i\\
&\succeq x_1\cdots x_kx_{k+1}^2\cdots x_{n+1}^2=\bq, &&(\text{by}~\eqref{05})
\end{align*}
where $\bw_i$, $\bp_i$, $\bp_{ij}$ may be empty for each $1\leq i \leq n+1$.
This yields $\bu \succeq \bq$.
\end{proof}

For $n\geq2$ and each admissible $k$, let
\[
\Upsilon_{n,k}
=
\bigl\{
\bq_{n,k}\preceq\bv
\mid
\bv\in\mathscr{U}_{n,k}
\bigr\}.
\]

\begin{thm}
The ai-semiring $\mathcal{P}^+(S_7, \cdot)$ is nonfinitely based.
\end{thm}
\begin{proof}
By Proposition~\ref{pro54501},
\[
\Xi=\{\eqref{001}, \eqref{003}, \eqref{004}, \eqref{006}\} \cup\{ \eqref{005}\}
\]
is an equational basis for the commutative ai-semiring $\mathcal{P}^+(S_7, \cdot)$, which contains $\Omega_1$.
By Theorem~\ref{thm02}, to show that $\mathcal{P}^+(S_7, \cdot)$ is nonfinitely based,
it suffices to prove that for any $m\geq 4$, $\bu_{m, 1}$ is $\bw$-free for every term $\bw$ in the sets
\[
\Gamma=\{\bt \mid \bt ~\text{is an upper side of an inequality in}~\{\eqref{001}, \eqref{003}, \eqref{004}, \eqref{006}\}\}
 \]
 and
 \[
\{\bv \mid \bv\in \mathscr{U}_{n, k}\setminus \{\bu_{n,1}\}, n\geq2\}.
\]
Indeed, the required freeness follows from Lemma~\ref{lem26011801}(d) for $\mathbf{w} \in\Gamma$
and from Proposition~\ref{inequality005} for $\mathbf{t}\in \{\bv \mid \bv\in \mathscr{U}_{n, k}\setminus \{\bu_{n,1}\}, n\geq2\}$.
\end{proof}

\section{The interval $[\mathsf{V}(\mathcal{P}^{+}(S_7,\cdot)), \mathsf{V}(\mathcal{P}(S_7,\cdot))]$}
In this section, we prove that the interval
$[\mathsf{V}(\mathcal{P}^{+}(S_7,\cdot)), \mathsf{V}(\mathcal{P}(S_7,\cdot))]$,
which consists of all subvarieties of $\mathsf{V}(\mathcal{P}(S_7,\cdot))$ that include $\mathsf{V}(\mathcal{P}^{+}(S_7,\cdot))$,
contains $2^{\aleph_0}$ distinct varieties.
To better understand the interval,
we first provide an equational characterization of its lower end $\mathsf{V}(\mathcal{P}^{+}(S_7,\cdot))$ in the upper end $\mathsf{V}(\mathcal{P}(S_7,\cdot))$.

\begin{pro}\label{lemqj}
$\mathsf{V}(\mathcal{P}^{+}(S_7,\cdot))$ is the subvariety of $\mathsf{V}(\mathcal{P}(S_7,\cdot))$ defined by the inequalities
\eqref{004}, \eqref{006} and \eqref{005}.
\end{pro}
\begin{proof}
Its proof follows from the equational bases of $\mathsf{V}(\mathcal{P}^{+}(S_7,\cdot))$ and $\mathsf{V}(\mathcal{P}(S_7,\cdot))$.
\end{proof}

Let $\mathbb{N}_{\geq 4}$ denote the set of all positive integers at least $4$.
For each subset $M$ of $\mathbb{N}_{\geq 4}$,
let $\mathcal{V}_M$ denote the subvariety of $\mathsf{V}(\mathcal{P}(S_7,\cdot))$
defined by the identities $\sigma_{m, 1}$ with $m\in M$.
Since every inequality of the form $\sigma_{m, 1}$ holds in $\mathsf{V}(\mathcal{P}^{+}(S_7,\cdot))$,
each $\mathcal{V}_M$ lies in the interval $[\mathsf{V}(\mathcal{P}^{+}(S_7,\cdot)), \mathsf{V}(\mathcal{P}(S_7,\cdot))]$.

\begin{thm}\label{thmqj}
The interval $[\mathsf{V}(\mathcal{P}^{+}(S_7,\cdot)), \mathsf{V}(\mathcal{P}(S_7,\cdot))]$ contains $2^{\aleph_0}$ distinct varieties.
\end{thm}
\begin{proof}
For any two subsets $P$ and $Q$ of $\mathbb{N}_{\geq 4}$,
we claim that $\mathcal{V}_P=\mathcal{V}_Q$ if and only if $P=Q$.
Suppose that $P\neq Q$; without loss of generality let $p \in P \setminus Q$.
It is evident that $\mathcal{V}_P$ satisfies $\sigma_{p, 1}$.
To see that $\mathcal{V}_Q$ does not satisfy $\sigma_{p, 1}$, recall that $\mathcal{V}_Q$ is defined by
the identities \eqref{01}--\eqref{05} together with all $\sigma_{q,1}\;(q \in Q)$.
By Lemmas~\ref{lem26012701}, \ref{lem02} and \ref{lem26011801}
Propositions~\ref{pro63401}, \ref{pro01} and \ref{pro02},
the upper side of $\sigma_{p,1}$
is free for the upper side of each of the defining inequalities of $\mathcal{V}_Q$.
Consequently, $\sigma_{p,1}$ cannot hold in $\mathcal{V}_Q$.
Therefore, $\mathcal{V}_P \neq \mathcal{V}_Q$.

Thus $\{\mathcal{V}_N \mid N\subseteq \mathbb{N}_{\geq 4}\}$ has the same cardinality as the power set of $\mathbb{N}_{\geq 4}$, and so
the interval $[\mathsf{V}(\mathcal{P}^{+}(S_7,\cdot)), \mathsf{V}(\mathcal{P}(S_7,\cdot))]$ contains $2^{\aleph_0}$ distinct varieties.
\end{proof}

\begin{cor}
Let $M$ be an infinite subset of $\mathbb{N}_{\geq 4}$.
Then the variety $\mathcal{V}_M$ is nonfinitely based.
\end{cor}
\begin{proof}
By Proposition~\ref{pro63401}, the set
\[
\Sigma=\{\eqref{01}, \eqref{03}, \eqref{04}, \eqref{05}\}
\cup\{\sigma_{m,1} \mid m \in M\}
\]
is an equational basis for the commutative ai-semiring variety $\mathcal{V}_M$,
which contains an infinite subset of $\Omega_1$.
According to Theorem~\ref{thm02}, it suffices to prove that
$\bu_{m, 1}$ is $\bt$-free for every inequality $\bs \preceq \bt$
in $\Sigma \setminus \Omega_1$ and for every $m\geq 4$.
Observe that
\[
\Sigma \setminus \Omega_1=\{\eqref{01}, \eqref{03}, \eqref{04}, \eqref{05}\}.
\]
The required freeness follows from Lemma~\ref{lem26011801}(d) for $\mathbf{t} \in\Gamma$, where
\[
\Gamma=\{\bw \mid \bw ~\text{is an upper side of an inequality in}~\{\eqref{01}, \eqref{03}, \eqref{04}, \eqref{05}\setminus\delta_{n, \bv}\}
\]
and from Proposition~\ref{pro02} for $\mathbf{t}\in \{\bv \mid \bv\in \Theta_n, n\geq 2\}$.
\end{proof}

\end{document}